\documentclass[submission,copyright,creativecommons]{eptcs}
\providecommand{\event}{ACT 2023} 

\usepackage{underscore}         
\usepackage[T1]{fontenc}        

\usepackage[shortlabels]{enumitem}

\usepackage{mathtools, amssymb, amsthm}
\usepackage[mathcal]{eucal}

\usepackage{tikz}
\usepackage{tikz-cd}

\newcommand{\E}{\mathcal{E}}
\newcommand{\M}{\mathcal{M}}
\newcommand{\C}{\mathcal{C}}
\newcommand{\W}{\mathcal{W}}
\newcommand{\JR}{\textsf{JR}}
\newcommand{\Alg}{\mathrm{\mathcal{A}lg}}
\newcommand{\JRAlg}{\mathrm{\mathcal{A}lg}_{\mathsf{JR}}}

\newcommand{\Lens}{\mathrm{\mathcal{L}ens}}
\newcommand{\Cat}{\mathrm{\mathcal{C}at}}
\newcommand{\CAT}{\mathrm{\mathcal{C}AT}}
\newcommand{\Set}{\mathrm{\mathcal{S}et}}

\mathchardef\mhyphen="2D

\DeclareMathOperator{\cod}{cod}

\usepackage{amsthm}

\newtheorem{theorem}{Theorem}
\newtheorem{proposition}[theorem]{Proposition}
\newtheorem{lemma}[theorem]{Lemma}
\newtheorem{corollary}[theorem]{Corollary}

\theoremstyle{definition} 
\newtheorem{definition}[theorem]{Definition}
\newtheorem{example}[theorem]{Example}
\newtheorem{notation}[theorem]{Notation}

\theoremstyle{remark}
\newtheorem{remark}[theorem]{Remark}

\hypersetup{
    colorlinks=true,
    allcolors=[rgb]{0.5,0,0}
    }

\title{The Algebraic Weak Factorisation System for Delta Lenses}

\author{
	Bryce Clarke
	\institute{Inria Saclay\\ Palaiseau, France}
	\email{bryce.clarke@inria.fr}
	}
	
\def\titlerunning{The Algebraic Weak Factorisation System for Delta Lenses}
\def\authorrunning{B. Clarke}

\begin{document}
\maketitle

\begin{abstract}
Delta lenses are functors equipped with a suitable choice of lifts, 
and are used to model bidirectional transformations between systems. 
In this paper, we construct an algebraic weak factorisation system
whose $R$-algebras are delta lenses. 
Our approach extends a semi-monad for delta lenses previously 
introduced by Johnson and Rosebrugh, and 
generalises to any suitable category 
equipped with an orthogonal factorisation system 
and an idempotent comonad.
We demonstrate how the framework of an algebraic weak factorisation 
system provides a natural setting for understanding the lifting 
operation of a delta lens, and also present an explicit description 
of the free delta lens on a functor.
\end{abstract}

\section{Introduction}
\label{section:introduction}

Delta lenses were first introduced by Diskin, Xiong, and Czarnecki 
\cite{DXC11} as an algebraic framework for 
\emph{bidirectional transformations} \cite{ACGMS18, CFHLST09} 
between systems, particularly in the context of 
\emph{model-driven engineering} \cite{DM12, Sil15}. 
The original motivation behind delta lenses came from adapting
the classical notion of a lens \cite{FGMPS07} from a ``state-based'' 
setting to a ``delta-based'' setting.
Instead of treating a system as a mere \emph{set} of states, 
it should be regarded as a \emph{category}, whose objects are the states 
of the system and whose morphisms are the updates (or deltas) 
between them. 
The purpose of delta lenses is to model the notion of 
\emph{synchronisation} between systems through specifying how certain
updates between states are propagated. 

A delta lens is a functor 
$f \colon A \rightarrow B$ equipped with a \emph{lifting operation}, see~\eqref{equation:delta-lenses-lifting}, that satisfies certain axioms.
The lifting operation specifies, for each object $a$ in $A$ 
and for each morphism $u \colon fa \rightarrow b$ in $B$, 
a morphism $\varphi(a, u) \colon a \rightarrow a'$, often called
the \emph{chosen lift}, such that
$f\varphi(a, u) = u$. 
The axioms placed on the lifting operation ensure that it respects 
identities and composition.
Thus a delta lens is a functor equipped with additional 
\emph{algebraic structure}, and it is natural to wonder if 
delta lenses arise as algebras for a monad. 
In this paper, we provide an answer in the affirmative. 
\begin{equation}
\label{equation:delta-lenses-lifting}
	\begin{tikzcd}[/tikz/column 1/.append style={anchor=base west}]
	\{ 0 \}
	\arrow[r, "a"]
	\arrow[d, end anchor = 137]
	&
	A
	\arrow[d, "f"]
	\\
	\{0 \rightarrow 1 \}
	\arrow[r, "u"]
	\arrow[ru, dashed, "{\varphi(a,\, u)}" description, start anchor = north]
	&
	B
	\end{tikzcd}
\end{equation}

The question of asking whether certain kinds of lenses are algebras for 
a monad is not new. 
Classical state-based lenses \cite{FGMPS07} were 
characterised by Johnson, Rosebrugh, and Wood \cite{JRW10} as algebras
for a monad on the slice category $\Set / B$. 
The same authors later introduced the notion of a \emph{c-lens}
\cite{JRW12}, better known as a \emph{split opfibration}, and 
characterised them as algebras for a monad, first introduced by 
Street~\cite{Str74}, on the slice category $\Cat / B$.
Delta lenses generalise state-based lenses and split opfibrations~\cite{JR16}, 
however they were only shown by Johnson and Rosebrugh \cite{JR13} 
to be certain algebras for a \emph{semi-monad} 
(a monad without a unit) on $\Cat / B$. 
One of the contributions of the current paper is resolve this 
gap in the literature.

Although it is generally useful to know when a mathematical 
structure arises as an algebra for a monad, in isolation
this result provides limited benefit 
towards a deeper understanding of lenses.
One reason is that we wish to study lenses as 
\emph{morphisms} of a category, 
rather than \emph{objects} in a category of algebras.
The knowledge that lenses are morphisms with algebraic structure 
does not provide any information of how to sequentially
\emph{compose} them, 
nor justification for why this algebraic structure encodes 
a notion of \emph{lifting}. 

Cofunctors\footnote{The term \emph{retrofunctor} proposed by Di Meglio \cite{DiM21} is preferred, but not yet in widespread use.} 
are a natural kind of morphism between categories \cite{Agu97, HM93} 
which fundamentally involve a lifting operation
and admit a straightforward sequential composition.  
The characterisation of delta lenses as a compatible functor 
and cofunctor \cite{AU17, Cla20}, together with related 
characterisations of state-based lenses and split opfibrations 
\cite{TAC}, 
provides a clear understanding of their composition and lifting, 
and has led to several fruitful developments in the study 
of lenses in applied category theory \cite{CCJSWZ22, ACT2020, DiM21}.
However the question remains: why do lenses frequently arise
as algebras for a monad? 

An \emph{algebraic weak factorisation system} \cite{BG16}, 
also known as a \emph{natural weak factorisation system} \cite{GT06},
generalises the notion of an orthogonal factorisation system 
(\textsc{ofs}) on a category. 
An algebraic weak factorisation system (\textsc{awfs}) on a 
category $\C$ consists of a comonad $(L, \epsilon, \Delta)$ 
and a monad $(R, \eta, \mu)$ on~$\C^{\mathbf{2}}$ 
that are suitably compatible. 
The categories of \emph{$L$-coalgebras} and \emph{$R$-algebras} 
of an \textsc{awfs} $(L, R)$ 
replace the usual \emph{left} and \emph{right} classes of morphisms of 
an \textsc{ofs}. 
In particular, every morphism
factors into a cofree $L$-coalgebra followed by free $R$-algebra,
and every lifting problem \eqref{equation:lifting-problem},
where $(f, p)$ is a $L$-coalgebra and $(g, q)$ is a $R$-algebra,
admits a chosen lift 
$\varphi_{f, g}\langle h, k \rangle$ making the diagram commute.
Crucially, these chosen lifts also induce a canonical composition 
of $R$-algebras \cite[Section~2.8]{BG16}.
Both classical state-based lenses 
and split opfibrations 
arise as $R$-algebras for an \textsc{awfs} on $\Set$ and $\Cat$, respectively.
\begin{equation}
\label{equation:lifting-problem}
	\begin{tikzcd}[column sep = large, row sep = large]
	A
	\arrow[r, "h"]
	\arrow[d, "{(f,\, p)}"']
	&
	C
	\arrow[d, "{(g,\, q)}"]
	\\
	B
	\arrow[r, "k"']
	\arrow[ru, dashed, "{\varphi_{\,f,\, g}\langle h,\, k \rangle}" description]
	&
	D
	\end{tikzcd}
\end{equation}

The main contribution of this paper is to construct an algebraic 
weak factorisation system $(L, R)$ on $\Cat$ whose $R$-algebras are 
precisely delta lenses. 
The principal benefit is a new framework for understanding lenses as 
algebras for a monad that naturally incorporates the fundamental 
aspects of composition and lifting. 
In addition, we are able to generalise the notion of 
delta lens to any suitable category equipped with an 
orthogonal factorisation system and idempotent comonad, 
as well as present an explicit description of the free delta lens 
on the functor.
This approach to lenses as algebras for a monad also 
highlights an interesting duality with their recent 
characterisation as coalgebras for a comonad \cite{Cla21}.

\textbf{Overview of the paper.} 
In Section~\ref{section:background} we review the necessary 
background material on delta lenses and factorisation systems. 
In particular, we recall two important structures on $\Cat$, 
the \emph{comprehensive factorisation system} 
(Example~\ref{example:comprehensive-factorisation-system})
and the \emph{discrete category comonad}
(Example~\ref{example:boo}), 
which are generalised in our main constructions 
to an orthogonal factorisation system and an idempotent comonad,
respectively.
In Section~\ref{section:algebras-for-a-semi-monad} 
we utilise these structures on a category $\C$ 
to build a semi-monad on $\C^{\mathbf{2}}$ 
(Proposition~\ref{proposition:semi-monad}), 
and show that when $\C = \Cat$ (Example~\ref{example:JR-algebras}), 
we recover delta lenses as certain algebras for this semi-monad
(Theorem~\ref{theorem:JR-isomorphism} and Appendix~\ref{section:appendix}). 
In Section~\ref{section:algebras-for-a-monad} we  
enhance this construction to a monad 
(Theorem~\ref{theorem:monad-for-lenses}) using pushouts in $\C$, 
and prove that when $\C = \Cat$, the algebras for this monad 
are delta lenses 
(Theorem~\ref{theorem:delta-lenses-are-algebras}).
We also describe the free delta lens on a functor (Example~\ref{example:free-delta-lens}). 
Section~\ref{section:awfs} completes the construction of an 
algebraic weak factorisation system on $\C$ 
(Theorem~\ref{theorem:awfs}) 
and shows how delta lenses lift against the $L$-coalgebras
when $\C = \Cat$.
Section~\ref{section:conclusion} presents some concluding remarks
and avenues for future work.

\section{Background}
\label{section:background}

\subsection{Delta lenses} 

We introduce the category $\Lens$ whose objects are delta lenses, 
which we will later show is the category of algebras for a monad
on $\Cat^{\mathbf{2}}$. 
For further details and examples, 
we refer the reader to \cite[Chapter~2]{Cla22}.

\begin{definition}
\label{definition:delta-lens}
A \emph{delta lens} $(f, \varphi) \colon A \rightarrow B$ consists of 
a functor $f \colon A \rightarrow B$ together with a 
\emph{lifting operation}
\[
	(a \in A, u \colon fa \rightarrow b \in B) 
	\qquad \longmapsto \qquad
	\varphi(a, u) \colon a \rightarrow p(a, u),
\]
where $p(a, u) = \cod(\varphi(a, u))$, that satisfies the following 
three axioms: 
\begin{enumerate}[(L1)]
    \item \label{(L1)} $f\varphi(a, u) = u$ 
    \item \label{(L2)} $\varphi(a, 1_{fa}) = 1_{a}$
    \item \label{(L3)} $\varphi(a, v \circ u) = \varphi(p(a, u), v) \circ \varphi(a, u)$
\end{enumerate}
\end{definition}

\begin{example}
\label{example:discrete-opfibration}
A \emph{discrete opfibration} is a functor $f \colon A \rightarrow B$
such that for each pair $(a \in A, u \colon fa \rightarrow b \in B)$ 
there is a unique morphism 
$\bar{f}(a, u) \colon a \rightarrow a'$ in $A$
for which $f\bar{f}(a, u) = u$. 
Thus each discrete opfibration $f$ admits a unique lifting operation 
$\bar{f}$ such that the pair $(f, \bar{f}\,)$ is a delta lens.
Conversely, the underlying functor $f$ of a delta lens $(f, \varphi)$ 
is a discrete opfibration if $\varphi(a, fw) = w$ for all morphisms
$w \colon a \rightarrow a'$ in~$A$. 
\end{example}

\begin{definition}
\label{definition:category-Lens}
Let $\Lens$ denote the category whose objects are delta lenses 
and whose morphisms
$\langle h, k \rangle$ from $(f, \varphi)$ to $(g, \psi)$ 
consist of a pair of functors $h$ and $k$ such that 
$k \circ f = g \circ h$ and $h\varphi(a, u) = \psi(ha, ku)$. 
\begin{equation*}
    \begin{tikzcd}
    A 
    \arrow[r, "h"]
    \arrow[d, "{(f,\, \varphi)}"']
    &
    C
    \arrow[d, "{(g,\, \psi)}"]
    \\
    B
    \arrow[r, "k"']
    &
    D
    \end{tikzcd}
\qquad \leftrightsquigarrow \qquad 
	\begin{tikzcd}[/tikz/column 1/.append style={anchor=base west}]
	\{ 0 \}
	\arrow[r, "a"]
	\arrow[d, end anchor = 137]
	&
	A
	\arrow[d, "f"]
	\arrow[r, "h"]
	& 
	C
	\arrow[d, "g"]
	\\[-5pt]
	\{0 \rightarrow 1 \}
	\arrow[r, "u"]
	\arrow[ru, dashed, "{\varphi(a,\, u)}" description, start anchor = north]
	&
	B
	\arrow[r, "k"]
	&
	D
	\end{tikzcd}
\quad = \quad 
	\begin{tikzcd}[/tikz/column 1/.append style={anchor=base west}]
	\{ 0 \}
	\arrow[r, "ha"]
	\arrow[d, end anchor = 137]
	&
	C
	\arrow[d, "g"]
	\\[-5pt]
	\{0 \rightarrow 1 \}
	\arrow[r, "ku"]
	\arrow[ru, dashed, "{\psi(ha,\, ku)}" description, start anchor = north]
	&
	D
	\end{tikzcd}
\end{equation*}
\end{definition}

Let $U \colon \Lens \rightarrow \Cat^{\mathbf{2}}$ denote the canonical
forgetful functor that sends $(f, \varphi)$ to $f$. 

\subsection{Factorisation systems}

We recall two related notions of factorisation system on
a category: \emph{orthogonal factorisations systems} \cite{FK72} and 
\emph{algebraic weak factorisation systems} \cite{BG16, GT06}.
For a full account, we refer the reader to \cite{BG16}.

\begin{definition}
\label{definition:ofs}
An \emph{orthogonal factorisation system} $(\E, \M)$ on a category $\C$
consists of two classes of morphisms $\E$ and $\M$, both containing the 
isomorphisms and closed under composition, such that: 
\begin{enumerate}[(i)]
\item \textbf{Factorisation}: Every morphism $f$ of $\C$ admits a factorisation 
$f = m \circ e$ with $e \in \E$ and $m \in \M$;
\item \textbf{Orthogonality}: For each solid commutative square in $\C$ 
below such that $e \in \E$ and $m \in \M$, 
there exists a unique morphism 
$h$ such that $f = h \circ e$ and $g = m \circ h$. 
\begin{equation*}
	\begin{tikzcd}
	A
	\arrow[r, "f"]
	\arrow[d, "e"']
	&
	C
	\arrow[d, "m"]
	\\
	B
	\arrow[r, "g"']
	\arrow[ru, dashed, "\exists!", "h"']
	& 
	D
	\end{tikzcd}
\end{equation*}
\end{enumerate}
\end{definition}

\begin{notation}
\label{notation:ofs-decoration}
As an aid when diagram-chasing, 
the morphisms in the left class~$\E$ and the right class~$\M$ of an 
orthogonal factorisation system on $\C$ 
will be decorated in the remainder of the paper as follows. 
\begin{equation*}
	\begin{tikzcd}[column sep = large]
	\bullet
	\arrow[r, two heads, "e\, \in\, \E"]
	&
	\bullet
	\end{tikzcd}
\qquad \qquad
	\begin{tikzcd}[column sep = large]
	\bullet
	\arrow[r, tail, "m\, \in\, \M"]
	&
	\bullet
	\end{tikzcd}
\end{equation*}
\end{notation}

\begin{example}
\label{example:comprehensive-factorisation-system}
A functor $f \colon A \rightarrow B$ is called \emph{initial} 
if, for each object $b \in B$, the comma category $f / b$ is connected.
The \emph{comprehensive factorisation system}~\cite{SW73} 
is an orthogonal factorisation system $(\E, \M)$ on $\Cat$
in which $\E$ is the class of initial functors and $\M$ is the class
of discrete opfibrations.
\end{example}

\begin{lemma}
\label{lemma:closure-properties}
If $(\E, \M)$ be an orthogonal factorisation system on $\C$,
then the following properties hold: 
\begin{enumerate}[(1)]
\item The class $\E$ is stable under pushouts in $\C$.
\item If $g \circ f$ and $f$ are in $\E$, then $g$ is in $\E$. 
Dually, if $g \circ f$ and $g$ are in $\M$, then $f$ is in $\M$. 
\end{enumerate}
\end{lemma}

\begin{definition}
\label{definition:functorial-factorisation}
A \emph{functorial factorisation} $(L, E, R)$ on a category $\C$ is a 
section $(L, E, R) \colon \C^{\mathbf{2}} \rightarrow \C^{\mathbf{3}}$ 
to the composition functor 
$\C^{\mathbf{3}} \rightarrow \C^{\mathbf{2}}$. 
The factorisation of a morphism in 
$\C^{\mathbf{2}}$ is denoted as follows. 
\begin{equation*}
	\begin{tikzcd}
	A
	\arrow[r, "h"]
	\arrow[d, "f"']
	&
	C
	\arrow[d, "g"]
	\\
	B
	\arrow[r, "k"']
	& 
	D
	\end{tikzcd}
\qquad \longmapsto \qquad
	\begin{tikzcd}
	A
	\arrow[r, "h"]
	\arrow[d, "Lf"']
	\arrow[dd, bend right = 50, "f"']
	&
	C
	\arrow[d, "Lg"]
	\arrow[dd, bend left = 50, "g"]
	\\
	Ef
	\arrow[r, "{E\langle h,\, k \rangle}"]
	\arrow[d, "Rf"']
	& 
	Eg
	\arrow[d, "Rg"]
	\\
	B
	\arrow[r, "k"']
	&
	D
	\end{tikzcd}
\end{equation*}
\end{definition}

\begin{remark}
\label{remark:functorial-factorisation}
Each functorial factorisation $(L, E, R)$ on $\C$ induces 
a copointed endofunctor $(L, \epsilon)$ 
and a pointed endofunctor $(R, \eta)$ on $\C^{\mathbf{2}}$, 
where the components of
$\epsilon \colon L \Rightarrow 1$ and 
$\eta \colon 1 \Rightarrow R$ at $f$
are given below. 
\begin{equation}
\label{equation:epsilon-eta}
	\begin{tikzcd}
	A
	\arrow[r, equal]
	\arrow[d, "Lf"']
	& 
	A
	\arrow[d, "f"]
	\\
	Ef
	\arrow[r, "Rf"']
	&
	B
	\end{tikzcd}
\qquad \qquad
	\begin{tikzcd}
	A
	\arrow[r, "Lf"]
	\arrow[d, "f"']
	& 
	Ef
	\arrow[d, "Rf"]
	\\
	B
	\arrow[r, equal]
	&
	B
	\end{tikzcd}
\end{equation}
\end{remark}

\begin{definition}
\label{definition:awfs}
\cite[Section~2.2]{BG16} 
An \emph{algebraic weak factorisation system} $(L, R)$ on a category 
$\C$ consists of:
\begin{enumerate}[(i)]
\item A functorial factorisation $(L, E, R)$ on $\C$;
\item An extension of $(L, \epsilon)$ to a comonad $(L, \epsilon, \Delta)$ on $\C^{\mathbf{2}}$;
\item An extension of $(R, \eta)$ to a monad $(R, \eta, \mu)$ on $\C^{\mathbf{2}}$;
\item A distributive law $\lambda \colon LR \Rightarrow RL$  
of the comonad $L$
over the monad $R$ with $\lambda_{f} = \langle \Delta_{f}, \mu_{f} \rangle$. 
\end{enumerate}
\end{definition}

\subsection{Idempotent comonads and weak equivalences} 

Given an idempotent comonad $(M, \iota)$ on a category $\C$,
let $\W = \{ f \in \C \mid \text{$Mf$ is invertible} \}$
denote the class of morphisms in $\C$ whose members are called 
\emph{weak equivalences}. 
This class satisfies the \emph{2-out-of-3 property}, 
and contains the isomorphisms, thus making $\C$ a 
\emph{category with weak equivalences} \cite{DHKS04}.
Since the comonad $M$ is idempotent, each counit 
component $\iota_{A}$ is inverted by $M$ and therefore 
a morphism of~$\W$. 
If $M$ preserves pushouts, the morphisms in $\W$ are stable under
pushout along morphisms in $\C$.

\begin{notation}
\label{notation:we-decoration}
As a visual aid when diagram-chasing, 
the morphisms in the class $\W$ of weak equivalences of a 
category $\C$ will be decorated in the remainder of the paper 
as follows. 
\begin{equation*}
	\begin{tikzcd}[column sep = large]
	\bullet 
	\arrow[r, "\sim"', "w\, \in\, \W"]
	&
	\bullet
	\end{tikzcd}
\end{equation*}
\end{notation}

\begin{example}
\label{example:boo}
Let $(-)_{0} \colon \Cat \rightarrow \Cat$ denote the idempotent 
comonad that assigns a category $A$ to its corresponding \emph{discrete
category} $A_{0}$ with counit component 
$\iota_{A} \colon A_{0} \rightarrow A$. 
The endofunctor $(-)_{0}$ has a right adjoint 
(the \emph{codiscrete category} monad) 
and therefore preserves all colimits.
A functor $f \colon A \rightarrow B$ is called 
\emph{bijective-on-objects} if $f_{0}$ is invertible;
these are the weak equivalences with respect to~$(-)_{0}$. 
\end{example}

\section{Delta lenses as certain algebras for a semi-monad}
\label{section:algebras-for-a-semi-monad}

Throughout this section, 
let $(\E, \M)$ be an orthogonal factorisation system on a category $\C$,
and let $(M, \iota)$ be an idempotent comonad on 
$\C$ with corresponding class $\W$ of weak equivalences.

\subsection{Constructing a semi-monad for delta lenses}

We now construct a semi-monad $(T, \nu)$ on the category 
$\C^{\mathbf{2}}$, for a category $\C$ equipped with an idempotent comonad $(M, \iota)$ and an orthogonal factorisation system 
$(\E, \M)$. 
We show that when $\C = \Cat$ equipped with the discrete category 
comonad and the comprehensive factorisation system, 
this specialises to the semi-monad defined on $\Cat^{\mathbf{2}}$ 
by Johnson and Rosebrugh \cite[Section~6]{JR13}.

We begin by constructing an endofunctor 
$T \colon \C^{\mathbf{2}} \rightarrow \C^{\mathbf{2}}$. 
Given a morphism $f \colon A \rightarrow B$ in $\C$, we first 
pre-compose with the counit component 
$\iota_{A} \colon MA \rightarrow A$ and then choose an 
$(\E, \M)$-factorisation of the resulting morphism as depicted 
in commutative square~(i) below; this defines the action of $T$ on 
objects in $\C^{\mathbf{2}}$. 
Given a morphism $\langle h, k \rangle \colon f \rightarrow g$ in
$\C^{\mathbf{2}}$, there exists a unique
morphism $J \langle h, k \rangle \colon Jf \rightarrow Jg$ in $\C$ 
by applying the orthogonality property; 
the action of $T$ on the morphism $\langle h, k \rangle$ is given by 
the commutative square~(ii) depicted below. 
Note that the equation \eqref{equation:endofunctor-T} holds by 
naturality of $\iota \colon M \Rightarrow 1$ at the morphism~$h$. 
\begin{equation}
\label{equation:endofunctor-T}
	\begin{tikzcd}
	MA
	\arrow[r, "\iota_{A}", "\sim"']
	\arrow[d, two heads, "Sf"']
	\arrow[rdd, phantom, "\text{(i)}"]
	& 
	A
	\arrow[dd, "f"']
	\arrow[r, "h"]
	& 
	C
	\arrow[dd, "g"]
	\\
	Jf
	\arrow[d, tail, "Tf"']
	& 
	\\
	B
	\arrow[r, equal]
	&
	B
	\arrow[r, "k"']
	& 
	D
	\end{tikzcd}
\qquad = \qquad
	\begin{tikzcd}
	MA
	\arrow[d, two heads, "Sf"']
	\arrow[r, "Mh"]
	&
	MC
	\arrow[d, two heads, "Sg"]
	\arrow[r, "\iota_{C}", "\sim"']
	&
	C
	\arrow[dd, "g"]
	\\
	Jf
	\arrow[d, tail, "Tf"']
	\arrow[r, dashed, "{J\langle h,\,k \rangle}"]
	\arrow[rd, phantom, "\text{(ii)}"]
	& 
	Jg
	\arrow[d, tail, "Tg"]
	& 
	\\
	B
	\arrow[r, "k"']
	& 
	C
	\arrow[r, equal]
	&
	C
	\end{tikzcd}
\end{equation}

Applying the functor $T$ to the morphism 
$Tf \colon Jf \rightarrow B$ and using the orthogonality property, 
we obtain the component $\nu_{f}$ of the multiplication 
$\nu \colon T^{2} \Rightarrow T$ at $f$ as depicted in the commutative
square~(iii) below. 
Naturality of $\nu$ at follows from noticing in 
\eqref{equation:semi-monad-naturality} that 
$J \langle h, k \rangle \circ \nu_{f} 
= \nu_{g} \circ  J\langle J\langle h, k \rangle, k \rangle$ by orthogonality.
\begin{equation}
\label{equation:semi-monad-naturality}
	\begin{tikzcd}
	MJf
	\arrow[d, two heads, "STf"']
	\arrow[r, "\iota_{Jf}", "\sim"']
	& 
	Jf
	\arrow[dd, tail, "Tf"]
	\arrow[r, "{J\langle h,\,k \rangle}"]
	&
	Jg
	\arrow[dd, tail, "Tg"]
	\\
	JTf
	\arrow[d, tail, "T^{2}f"']
	\arrow[ru, tail, dashed, "\nu_{f}"']
	\arrow[rd, phantom, "\text{(iii)}"{yshift=5pt}]
	& 
	&
	\\
	B
	\arrow[r, equal]
	&
	B
	\arrow[r, "k"']
	&
	D
	\end{tikzcd}
\qquad = \qquad
	\begin{tikzcd}[column sep = large]
	MJf
	\arrow[d, two heads, "STf"']
	\arrow[r, "{MJ\langle h,\,k \rangle}"]
	&
	MJg
	\arrow[r, "\iota_{Jg}", "\sim"']
	\arrow[d, two heads, "STg"']
	& 
	Jg
	\arrow[dd, tail, "Tg"]
	\\
	JTf
	\arrow[r, dashed, "{J\langle J\langle h,\, k \rangle,\, k \rangle}"]
	\arrow[d, tail, "T^{2}f"']
	&
	JTg
	\arrow[d, tail, "T^{2} g"']
	\arrow[ru, dashed, tail, "\nu_{g}"']
	& 
	\\
	B
	\arrow[r, "k"']
	&
	D
	\arrow[r, equal]
	&
	D
	\end{tikzcd}
\end{equation}
The associative law for $\nu$ follows from 
observing in \eqref{equation:semi-monad-associativity} that 
$\nu_{f} \circ \nu_{Tf} = \nu_{f} \circ J\langle \nu_{f}, 1_{B} \rangle$
by orthogonality. 
\begin{equation}
\label{equation:semi-monad-associativity}
	\begin{tikzcd}
	MJTf
	\arrow[d, two heads, "ST^{2}f"']
	\arrow[r, "\iota_{JTf}", "\sim"']
	& 
	JTf
	\arrow[dd, tail, "T^{2}f"]
	\arrow[r, tail, "\nu_{f}"]
	&
	Jf
	\arrow[dd, tail, "Tf"]
	\\
	JT^{2}f
	\arrow[d, tail, "T^{3}f"']
	\arrow[ru, tail, dashed, "\nu_{Tf}"']
	& 
	&
	\\
	B
	\arrow[r, equal]
	&
	B
	\arrow[r, equal]
	&
	B
	\end{tikzcd}
\qquad = \qquad
	\begin{tikzcd}[column sep = large]
	MJTf
	\arrow[d, two heads, "ST^{2}f"']
	\arrow[r, "M\nu_{f}"]
	&
	MJf
	\arrow[r, "\iota_{Jf}", "\sim"']
	\arrow[d, two heads, "STf"']
	& 
	Jf
	\arrow[dd, tail, "Tf"]
	\\
	JT^{2}f
	\arrow[r, tail, dashed, "{J\langle \nu_{f},\, 1_{B} \rangle}"]
	\arrow[d, tail, "T^{3}f"']
	&
	JTf
	\arrow[d, tail, "T^{2} f"']
	\arrow[ru, dashed, tail, "\nu_{f}"']
	& 
	\\
	B
	\arrow[r, equal]
	&
	B
	\arrow[r, equal]
	&
	B
	\end{tikzcd}
\end{equation}

We have thus constructed an endofunctor 
$T \colon \C^{\mathbf{2}} \rightarrow \C^{\mathbf{2}}$ 
with an associative multiplication $\nu \colon T^{2} \Rightarrow T$.

\begin{proposition}
\label{proposition:semi-monad}
The pair $(T, \nu)$ is a semi-monad on $\C^{2}$. 
\end{proposition}

\begin{corollary}
The semi-monad $(T, \nu)$ on $\C^{\mathbf{2}}$ restricts to a 
semi-monad in the $2$-category $\CAT / \C$ 
on the codomain functor $\cod \colon \C^{\mathbf{2}} \rightarrow \C$. 
In particular, $(T, \nu)$ induces a semi-monad on each slice 
category $\C / B$. 
\end{corollary}

\begin{example}
\label{example:JR-algebras}
Consider the category $\Cat$ equipped with the comprehensive factorisation system and the discrete category comonad. 
Given a functor $f \colon A \rightarrow B$,
the category $Jf$ defined in \eqref{equation:endofunctor-T} 
is given by the coproduct $\sum_{a \in A_{0}} fa / B$ of the coslice
categories indexed by the discrete category $A_{0}$. 
The objects in $Jf$ are pairs 
$(a \in A, u \colon fa \rightarrow b \in B)$,
while morphisms 
$\langle 1_{a}, v \rangle \colon (a, u_{1}) \rightarrow (a, u_{2})$ 
are given by morphisms $v \in B$ such that $u_{2} = v \circ u_{1}$. 
The functor $Sf \colon A_{0} \twoheadrightarrow J_{f}$ has an assignment
on objects $a \mapsto (a, 1_{fa})$, and is an \emph{initial functor}
since each slice category $Sf / (a, u)$ is isomorphic to the 
terminal category and hence connected.
The functor $Tf \colon Jf \rightarrowtail B$ is given by the codomain
projection with assignment on objects $(a, u) \mapsto \cod(u)$, and
is a \emph{discrete opfibration}. 
In this setting, 
restricting the semi-monad $(T, \nu)$ to the slice categories 
$\Cat/B$ coincides with semi-monad for delta lenses 
defined by Johnson and Rosebrugh~\cite{JR13}.
\end{example}

\subsection{Delta lenses as certain semi-monad algebras}

An \emph{algebra} $(f, p)$ 
for the semi-monad $(T, \nu)$ on the codomain functor 
$\cod \colon \C^{\mathbf{2}} \rightarrow \C$ 
(or, equivalently, on the slice category $\C / B$)
consists of a pair of morphisms 
$f \colon A \rightarrow B$ and $p \colon Jf \rightarrow A$ such that 
the following diagrams commute. 
\begin{equation}
\label{equation:semi-monad-algebra}
	\begin{tikzcd}
	Jf
	\arrow[r, "p"]
	\arrow[d, tail, "Tf"']
	&
	A
	\arrow[d, "f"]
	\\
	B
	\arrow[r, equal]
	& 
	B
	\end{tikzcd}
\qquad \qquad
	\begin{tikzcd}[column sep = large]
	JTf
	\arrow[r, "{J\langle p,\, 1_{B} \rangle}"]
	\arrow[d, tail, "\nu_{f}"']
	&
	Jf
	\arrow[d, "p"]
	\\
	Jf
	\arrow[r, "p"']
	& 
	A
	\end{tikzcd}
\end{equation}
Johnson and Rosebrugh (\JR)
introduced an additional condition on the algebras for the semi-monad 
$(T, \nu)$ on $\Cat / B$ which we now 
adapt to our more general setting under the name \emph{\JR-algebra}. 
The intuition is that this additional condition replaces the 
missing ``unit law'' that an algebra for a monad would satisfy. 

\begin{definition}
A \emph{\JR-algebra} is an algebra $(f, p)$ for the 
semi-monad $(T, \nu)$ on the codomain functor 
$\cod \colon \C^{\mathbf{2}} \rightarrow \C$ 
such that the following diagram commutes.
\begin{equation}
\label{equation:JR-algebra}
	\begin{tikzcd}
	MA
	\arrow[d, two heads, "Sf"']
	\arrow[r, "\iota_{A}", "\sim"']
	& 
	A
	\\
	Jf
	\arrow[ru, "p"']
	& 
	\end{tikzcd}
\end{equation}
\end{definition}

A \emph{morphism} 
$\langle h, k \rangle \colon (f, p) \rightarrow (g, q)$ 
of algebras for the semi-monad $(T, \nu)$ 
consists of a pair of morphisms 
$h$ and $k$ such that the following equation in $\C^{\mathbf{2}}$ holds. 
\begin{equation}
\label{equation:T-algebra-morphism}
	\begin{tikzcd}
	Jf
	\arrow[r, "p"]
	\arrow[d, tail, "Tf"']
	&
	A
	\arrow[d, "f"]
	\arrow[r, "h"]
	&
	C
	\arrow[d, "g"]
	\\
	B
	\arrow[r, equal]
	& 
	B
	\arrow[r, "k"']
	&
	D
	\end{tikzcd}
\qquad = \qquad
	\begin{tikzcd}
	Jf
	\arrow[r, "{J\langle h,\, k \rangle}"]
	\arrow[d, tail, "Tf"']
	& 
	Jg
	\arrow[r, "q"]
	\arrow[d, tail, "Tg"]
	&
	C
	\arrow[d, "g"]
	\\
	B
	\arrow[r, "k"']
	&
	D
	\arrow[r, equal]
	&
	D
	\end{tikzcd}
\end{equation}
Let $\Alg(T, \nu)$ denote the category of algebras for the semi-monad
$(T, \nu)$ on the codomain functor 
$\cod \colon \C^{\mathbf{2}} \rightarrow \C$, 
and let $\JRAlg(T, \nu)$ denote the full subcategory of \JR-algebras.

\begin{theorem}
\label{theorem:JR-isomorphism}
If $\C = \Cat$ equipped with the discrete category comonad and the 
comprehensive factorisation system, then
there is an isomorphism of categories $\Lens \cong \JRAlg(T, \nu)$.
\end{theorem}
\begin{proof}
This result is due to Johnson and Rosebrugh \cite{JR13}. 
See Appendix~\ref{section:appendix} for a proof in our notation. 
\end{proof}

\section{Delta lenses as algebras for a monad}
\label{section:algebras-for-a-monad}

Throughout this section, 
let $(\E, \M)$ be an orthogonal factorisation system on a category $\C$
with (chosen) pushouts, 
and let $(M, \iota)$ be an idempotent comonad on 
$\C$ such that $M \colon \C \rightarrow \C$ preserves pushouts. 

\subsection{Constructing a monad for delta lenses}

We now extend the semi-monad $(T, \nu)$ to a monad 
$(R, \eta, \mu)$ on $\C^{\mathbf{2}}$, for a category $\C$
as described above.
Our approach is to utilise the universal properties of pushouts 
and orthogonal factorisation systems, as well as properties of 
the class of weak equivalences for the idempotent comonad, 
to construct the necessary data 
for the monad from that of the semi-monad $(T, \nu)$. 

We begin by constructing an endofunctor 
$R \colon \C^{\mathbf{2}} \rightarrow \C^{\mathbf{2}}$. 
Given a morphism $f \colon A \rightarrow B$ in $\C$, 
first construct the pushout of $\iota_{A}$ along $Sf$ 
from \eqref{equation:endofunctor-T}, and then 
use the universal property of the pushout to define 
$Rf \colon Ef \rightarrow B$ as depicted on the left below;
this defines the action of $R$ on objects in~$\C^{\mathbf{2}}$. 
\begin{equation}
\label{equation:endofunctor-R-objects}
	\begin{tikzcd}
	MA
	\arrow[r, "\iota_{A}", "\sim"']
	\arrow[d, two heads, "Sf"']
	\arrow[rd, phantom, "\ulcorner", very near end]
	& 
	A
	\arrow[dd, bend left=50, "f"]
	\arrow[d, two heads, "Lf"'{yshift = 3pt}]
	\\
	Jf
	\arrow[d, tail, "Tf"']
	\arrow[r, "\alpha_{f}", "\sim"']
	& 
	Ef
	\arrow[d, dashed, "Rf"']
	\\
	B
	\arrow[r, equal]
	&
	B
	\end{tikzcd}
\qquad \qquad \qquad 
	\begin{tikzcd}
	MEf
	\arrow[r, "\iota_{Ef}", "\sim"']
	\arrow[d, two heads, "SRf"']
	\arrow[rd, phantom, "\ulcorner", very near end]
	& Ef
	\arrow[d, two heads, "LRf"'{yshift=3pt}]
	\arrow[dd, bend left=50, "Rf"]
	\\
	JRf
	\arrow[d, tail, "TRf"']
	\arrow[r, "\alpha_{Rf}", "\sim"']
	&
	ERf
	\arrow[d, dashed, "R^{2}f"']
	\\
	B
	\arrow[r, equal]
	&
	B
	\end{tikzcd}
\end{equation}
Given a morphism $\langle h, k \rangle \colon f \rightarrow g$ in
$\C^{\mathbf{2}}$, there exists a unique
morphism $E \langle h, k \rangle \colon Jf \rightarrow Jg$ in $\C$, 
as depicted below, by the universal property of the pushout, where 
$J\langle h, k \rangle$ is defined in \eqref{equation:endofunctor-T}. 
It is not difficult to show through diagram-chasing  that 
$Rg \circ E\langle h, k \rangle = k \circ Rf$,
thus defining the action of $R$ on morphisms of $\C^{\mathbf{2}}$.
\begin{equation}
\label{equation:endofunctor-R-morphisms}
\begin{tikzcd}
	MA
	\arrow[r, "\iota_{A}", "\sim"']
	\arrow[d, two heads, "Sf"']
	\arrow[rd, phantom, "\ulcorner", very near end]
	& 
	A
	\arrow[rd, "h"]
	\arrow[d, two heads, "Lf"'{yshift = 3pt}]
	& 
	\\
	Jf
	\arrow[rd, "{J\langle h,\, k \rangle}"']
	\arrow[r, "\alpha_{f}", "\sim"']
	& 
	Ef
	\arrow[rd, dashed, "{E\langle h,\, k \rangle}"{xshift = -6pt, yshift = 3pt}]
	&
	C
	\arrow[d, two heads, "Lg"]
	\\
	&
	Jg
	\arrow[r, "\alpha_{g}", "\sim"']
	& 
	Eg
	\end{tikzcd}
\qquad = \qquad
	\begin{tikzcd}
	MA
	\arrow[r, "\iota_{A}", "\sim"']
	\arrow[d, two heads, "Sf"']
	\arrow[rd, "Mh"{xshift=3pt, yshift = -2pt}]
	& 
	A
	\arrow[rd, "h"]
	& 
	\\
	Jf
	\arrow[rd, "{J\langle h,\, k \rangle}"']
	& 
	MC
	\arrow[r, "\iota_{C}", "\sim"']
	\arrow[d, two heads, "Sg"']
	\arrow[rd, phantom, "\ulcorner", very near end]
	&
	C
	\arrow[d, two heads, "Lg"]
	\\
	&
	Jg
	\arrow[r, "\alpha_{g}", "\sim"']
	& 
	Eg
	\end{tikzcd}
\end{equation}

\begin{lemma}
\label{lemma:functorial-factorisation}
The triple $(L, E, R)$ constructed in 
\eqref{equation:endofunctor-R-objects} and 
\eqref{equation:endofunctor-R-morphisms} is 
functorial factorisation on $\C$. 
\end{lemma}

By Remark~\ref{remark:functorial-factorisation}, this functorial
factorisation induces a pointed endofunctor $(R, \eta)$ on 
$\C^{\mathbf{2}}$ where the 
component of $\eta$ at $f$ is given by the morphism 
$Lf \colon A \rightarrow Ef$ as depicted in 
\eqref{equation:epsilon-eta}.
To extend this pointed endofunctor to a monad, all that remains is to
define a suitable multiplication $\mu \colon R^{2} \Rightarrow R$. 

To construct this multiplication, we first observe 
that the morphism $\alpha_{f} \colon Jf \rightarrow Ef$ constructed
in \eqref{equation:endofunctor-R-objects} is a weak equivalence, 
and therefore 
the morphism $M\alpha_{f} \colon MJf \rightarrow MEf$ 
is invertible. 
It follows from the orthogonality property that the morphism 
$J\langle \alpha_{f}, 1_{B} \rangle \colon JTf \rightarrow JRf$ 
is invertible as depicted below. 
\begin{equation}
\label{equation:useful-isomorphism}
	\begin{tikzcd}
	MJf
	\arrow[r, "M\alpha_{f}", "\cong"']
	\arrow[d, "\iota_{Jf}"', "\sim"{rotate=90, anchor=north}]
	&
	MEf
	\arrow[d, "\iota_{Ef}", "\sim"{rotate=90, anchor=south}]
	\\
	Jf
	\arrow[r, "\alpha_{f}", "\sim"']
	\arrow[d, tail, "Tf"']
	&
	Ef
	\arrow[d, "Rf"]
	\\
	B
	\arrow[r, equal]
	&
	B
	\end{tikzcd}
\qquad = \qquad
	\begin{tikzcd}[column sep = large]
	MJf
	\arrow[r, "M\alpha_{f}", "\cong"']
	\arrow[d, two heads, "STf"']
	&
	MEf
	\arrow[d, two heads, "SRf"]
	\\
	JTf
	\arrow[r, dashed, "{J\langle \alpha_{f},\, 1_{B} \rangle}", "\cong"']
	\arrow[d, tail, "T^{2}f"']
	&
	JRf
	\arrow[d, tail, "TRf"]
	\\
	B
	\arrow[r, equal]
	&
	B
	\end{tikzcd}
\end{equation}
Using the universal property of the pushout, the morphism 
$\nu_{f}$ defined in \eqref{equation:semi-monad-naturality},
and the morphism $J\langle \alpha_{f}, 1_{B} \rangle^{-1}$ defined 
in \eqref{equation:useful-isomorphism},
we obtain the 
component $\mu_{f}$ of the multiplication 
$\mu \colon R^{2} \Rightarrow R$ at $f$ as depicted below. 
\begin{equation}
\label{equation:multiplication-mu}
	\begin{tikzcd}
	MEf
	\arrow[r, "\iota_{Ef}", "\sim"']
	\arrow[d, two heads, "SRf"']
	\arrow[rd, phantom, "\ulcorner", very near end]
	& 
	Ef
	\arrow[rdd, equal, bend left]
	\arrow[d, two heads, "LRf"'{yshift = 3pt}]
	& 
	\\
	JRf
	\arrow[d, "{J\langle \alpha_{f},\, 1_{B} \rangle^{-1}}"', "\cong"]
	\arrow[r, "\alpha_{Rf}", "\sim"']
	& 
	ERf
	\arrow[rd, two heads, dashed, "\mu_{f}"]
	&
	\\
	JTf
	\arrow[r, tail, "\nu_{f}"]
	&
	Jf
	\arrow[r, "\alpha_{f}", "\sim"']
	& 
	Ef
	\end{tikzcd}
\qquad = \qquad
	\begin{tikzcd}
	MEf
	\arrow[r, "\iota_{Ef}", "\sim"']
	\arrow[d, two heads, "SRf"']
	\arrow[rd, "{(M\alpha_{f})^{-1}}"{yshift=-3pt, xshift=2pt}, "\cong"']
	& 
	Ef
	\arrow[rdd, equal, bend left]
	& 
	\\
	JRf
	\arrow[d, "{J\langle \alpha_{f},\, 1_{B} \rangle^{-1}}"', "\cong"]
	& 
	MJf
	\arrow[ld, two heads, "STf"'{yshift = 3pt, xshift = 6pt}]
	\arrow[d, "\iota_{Jf}", "\sim"'{rotate=90, anchor=center, yshift=3pt}]
	&
	\\
	JTf
	\arrow[r, tail, "\nu_{f}"]
	&
	Jf
	\arrow[r, "\alpha_{f}", "\sim"']
	& 
	Ef
	\end{tikzcd}
\end{equation}
A tedious, yet routine, exercise in diagram-chasing using the 
morphisms defined in \eqref{equation:endofunctor-R-morphisms} and \eqref{equation:multiplication-mu}, and applying the universal property of the pushout shows that $Rf \circ \mu_{f} = R^{2}f$
and that $\mu$ is natural as depicted below.
\begin{equation*}
	\begin{tikzcd}
	ERf
	\arrow[d, "R^{2}f"']
	\arrow[r, two heads, "\mu_{f}"]
	& 
	Ef
	\arrow[d, "Rf"]
	\arrow[r, "{E\langle h,\, k \rangle}"]
	&
	Eg
	\arrow[d, "Rg"]
	\\
	B
	\arrow[r, equal]
	& 
	B
	\arrow[r, "k"']
	&
	D
	\end{tikzcd}
\qquad = \qquad
	\begin{tikzcd}
	ERf
	\arrow[d, "R^{2}f"']
	\arrow[r, "{E\langle E\langle h,\, k \rangle,\, k \rangle}"]
	&[+1.5em]
	ERg
	\arrow[d, "R^{2}g"]
	\arrow[r, two heads, "\mu_{g}"]
	&
	Eg
	\arrow[d, "Rg"]
	\\
	B
	\arrow[r, "k"']
	& 
	D
	\arrow[r, equal]
	&
	D
	\end{tikzcd}
\end{equation*}
Showing that the diagrams below commute, and thus establishing that the 
multiplication $\mu$ 
is unital and associative, is also a straightforward
application of definitions and the universal property of the pushout.  
\begin{equation*}
	\begin{tikzcd}[column sep = large]
	Ef
	\arrow[r, two heads, "LRf"]
	\arrow[rd, equal]
	& 
	ERf
	\arrow[d, two heads, "\mu_{f}"]
	&
	Ef
	\arrow[l, two heads, "{E\langle Lf, 1_{B}\rangle}"']
	\arrow[ld, equal]
	\\
	& 
	Ef
	&
	\end{tikzcd}
\qquad \qquad
	\begin{tikzcd}[column sep = large]
	ER^{2}f
	\arrow[r, two heads, "{E\langle \mu_{f}, 1_{B}\rangle}"]
	\arrow[d, two heads, "\mu_{Rf}"']
	& ERf
	\arrow[d, two heads, "\mu_{f}"]
	\\
	ERf
	\arrow[r, two heads, "\mu_{f}"']
	&
	Ef
	\end{tikzcd}
\end{equation*}

\begin{theorem}
\label{theorem:monad-for-lenses}
The triple $(R, \eta, \mu)$ is a monad on $\C^{\mathbf{2}}$. 
\end{theorem}

\begin{corollary}
The monad $(R, \eta, \mu)$ on $\C^{\mathbf{2}}$ restricts to a 
monad in the $2$-category $\CAT / \C$ 
on the codomain functor $\cod \colon \C^{\mathbf{2}} \rightarrow \C$. 
In particular, $(R, \eta, \mu)$ induces a monad on each slice 
category $\C / B$. 
\end{corollary}

\begin{remark}
The morphisms $\alpha_{f}$ defined as pushout injections in 
\eqref{equation:endofunctor-R-objects} assemble into a natural 
transformation $\alpha \colon T \Rightarrow R$ which underlies a
morphism of semi-monads $(T, \nu) \rightarrow (R, \mu)$.
We conjecture that $(R, \eta, \mu)$ is actually the \emph{free monad} 
on the semi-monad $(T, \nu)$, in a suitable sense, however leave this for future work.
\end{remark}

\subsection{Delta lenses as monad algebras}

We now construct the algebras for the monad $(R, \eta, \mu)$ on
$\C^{\mathbf{2}}$ and show they are the same as \JR-algebras for 
the semi-monad $(T, \nu)$. 
When $\C = \Cat$ equipped with the comprehensive factorisation system and the discrete category comonad, this result establishes that 
delta lenses are algebras for the monad $(R, \eta, \mu)$. 

An \emph{algebra} 
$(f, \hat{p})$ for the monad 
$(R, \eta, \mu)$ on $\C^{\mathbf{2}}$ consists of a pair of morphisms 
$f \colon A \rightarrow B$ and 
$\hat{p} \colon Ef \rightarrow A$ 
such that the following diagrams commute: 
\begin{equation}
\label{equation:algebra-for-R}
	\begin{tikzcd}
	A
	\arrow[r, equal]
	\arrow[d, two heads, "Lf"']
	& 
	A
	\arrow[d, "f"]
	\\
	Ef
	\arrow[r, "Rf"']
	\arrow[ru, two heads, "\hat{p}"]
	&
	B
	\end{tikzcd}
\qquad \qquad
	\begin{tikzcd}[column sep = large]
	ERf
	\arrow[r, two heads, "{E\langle \hat{p},\, 1_{B}\rangle}"]
	\arrow[d, two heads, "\mu_{f}"']
	& 
	Ef
	\arrow[d, two heads, "\hat{p}"]
	\\
	Ef
	\arrow[r, two heads, "\hat{p}"']
	&
	A
	\end{tikzcd}
\end{equation}
A \emph{morphism} 
$\langle h, k \rangle \colon (f, \hat{p}) \rightarrow (g, \hat{q})$ 
of algebras for the monad $(R, \eta, \mu)$ 
consists of a pair of morphisms 
$h$ and $k$ such that the following equation in $\C^{\mathbf{2}}$ holds. 
\begin{equation}
\label{equation:R-algebra-morphism}
	\begin{tikzcd}
	Ef
	\arrow[r, two heads, "\hat{p}"]
	\arrow[d, "Rf"']
	&
	A
	\arrow[d, "f"]
	\arrow[r, "h"]
	&
	C
	\arrow[d, "g"]
	\\
	B
	\arrow[r, equal]
	& 
	B
	\arrow[r, "k"']
	&
	D
	\end{tikzcd}
\qquad = \qquad
	\begin{tikzcd}
	Ef
	\arrow[r, "{E\langle h,\, k \rangle}"]
	\arrow[d, "Rf"']
	& 
	Eg
	\arrow[r, two heads, "\hat{q}"]
	\arrow[d, "Rg"]
	&
	C
	\arrow[d, "g"]
	\\
	B
	\arrow[r, "k"']
	&
	D
	\arrow[r, equal]
	&
	D
	\end{tikzcd}
\end{equation}
Let $\Alg(R, \eta, \mu)$ denote the category of algebras for the monad 
$(R, \eta, \mu)$. 

\begin{proposition}
\label{proposition:JR-algebras-as-algebras}
There is an isomorphism of categories 
$\JRAlg(T, \nu) \cong \Alg(R, \eta, \mu)$. 
\end{proposition}
\begin{proof}
Let $(f \colon A \rightarrow B, p \colon Jf \rightarrow A)$ be a
$\JR$-algebra for the semi-monad $(T, \nu)$. 
Using the diagram \eqref{equation:JR-algebra} and the universal property
of the pushout, we obtain a morphism 
$[p, 1_{A}] \colon Ef \rightarrow A$ as depicted below. 
\begin{equation*}
	\begin{tikzcd}
	MA 
	\arrow[r, "\iota_{A}", "\sim"']
	\arrow[d, two heads, "Sf"']
	\arrow[rd, phantom, "\ulcorner", very near end]
	& 
	A
	\arrow[d, two heads, "Lf"]
	\arrow[rdd, bend left, "1_{A}"]
	&
	\\
	Jf
	\arrow[r, "\alpha_{f}", "\sim"']
	\arrow[rrd, bend right, "p"'{pos=0.3}]
	& 
	Ef
	\arrow[rd, dashed, "{[p, 1_{A}]}"{pos = 0.2}]
	&
	\\
	& 
	& 
	A
	\end{tikzcd}
\end{equation*}
Using the universal property of the pushout and the axioms for 
the \JR-algebra $(f, p)$, it is straightforward to prove that the pair
$(f, [p, 1_{A}])$ is an algebra for the monad $(R, \eta, \mu)$. 

Now consider an algebra 
$(f \colon A \rightarrow B, \hat{p} \colon Ef \rightarrow A)$ 
for the monad $(R, \eta, \mu)$. 
Pre-composing the structure map of the algebra with $\alpha_{f}$ 
we obtain a morphism $\hat{p} \circ \alpha_{f} \colon Jf \rightarrow B$. 
Using the axioms for the algebra $(f, \hat{p})$ and appropriate 
pasting of commutative diagrams, one may easily show that the pair
$(f, \hat{p} \circ \alpha_{f})$ is a \JR-algebra for the semi-monad 
$(T, \nu)$. 

The \JR-algebras for $(T, \nu)$ and the algebra for $(R, \eta, \mu)$ 
are in bijective correspondence with each other, since 
$[\hat{p} \circ \alpha_{f}, 1_{A}] = \hat{p}$ by the universal 
property of the pushout, and $[p, 1_{A}] \circ \alpha_{f} = p$ by 
construction. 
One may extend this correspondence to the morphisms 
\eqref{equation:T-algebra-morphism} and
\eqref{equation:R-algebra-morphism} of 
the respective categories and show it is functorial, thus
establishing the stated isomorphism of categories. 
\end{proof}

The following theorem establishes a key result of the paper: 
delta lenses are algebras for a monad. 

\begin{theorem}
\label{theorem:delta-lenses-are-algebras}
There is an isomorphism of categories $\Lens \cong \Alg(R, \eta, \mu)$. 
\end{theorem}
\begin{proof}
Follows directly from Theorem~\ref{theorem:JR-isomorphism} and 
interpreting Proposition~\ref{proposition:JR-algebras-as-algebras} 
in the setting of $\C = \Cat$ equipped with the discrete category 
comonad and the comprehensive factorisation system.
\end{proof}

\begin{corollary}
The forgetful functor $U \colon \Lens \rightarrow \Cat^{\mathbf{2}}$
is strictly monadic. 
\end{corollary}

\subsection{The free delta lens on a functor}

We now construct a left adjoint to the functor 
$U \colon \Lens \rightarrow \Cat^{\mathbf{2}}$ which defines the 
\emph{free delta lens} on a functor $f \colon A \rightarrow B$. 
This amounts to providing an explicit description of the category 
$Ef$ together with a lifting operation on the functor 
$Rf \colon Ef \rightarrow B$.
First we recall \cite[Corollary~20]{Cla20} the following result 
which represents an delta lens as a certain commutative diagram
(see \cite[Section~2.4]{Cla22} for a detailed proof). 

\begin{proposition}
\label{proposition:delta-lens-representation}
Each delta lens $(f, \varphi) \colon A \rightarrow B$ determines 
a commutative diagram in $\Cat$, as depicted on
the left below, such that $\varphi$ is bijective-on-objects and 
$f\,\varphi$ is a discrete opfibration. 
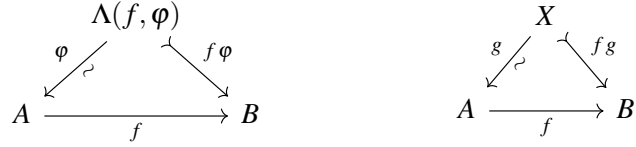
\begin{equation*}
	\begin{tikzcd}[column sep = small]
	& 
	\Lambda(f, \varphi)
	\arrow[ld, "\varphi"', "\sim"{rotate=45, anchor=north, pos = 0.4}]
	\arrow[rd, tail, "f\,\varphi"]
	&
	\\
	A
	\arrow[rr, "f"']
	& & 
	B
	\end{tikzcd}
\qquad \qquad \qquad
	\begin{tikzcd}[column sep = small]
	& 
	X
	\arrow[ld, "g"', "\sim"{rotate=45, anchor=north, pos = 0.4}]
	\arrow[rd, tail, "f\,g"]
	&
	\\
	A
	\arrow[rr, "f"']
	& & 
	B
	\end{tikzcd}
\end{equation*}
Conversely, each commutative diagram on the right above, 
where $g$ is bijective-on-objects and $fg$ is a discrete opfibration, 
uniquely determines a delta lens structure on $f$. 
\end{proposition}

\begin{remark}
The above result may be understood as a consequence 
of an equivalence of double categories \cite[Section~3.4]{Cla22}, 
however the details are outside the scope of this paper. 
\end{remark}

Using Proposition~\ref{proposition:delta-lens-representation},
the \emph{free delta lens} on a functor $f \colon A \rightarrow B$ 
corresponds to the following commutative diagram 
in $\Cat$ constructed in \eqref{equation:endofunctor-R-objects}. 
An immediate benefit of this presentation of the free delta lens
is that it condenses the three commutative diagrams 
\eqref{equation:algebra-for-R} for the (free) $R$-algebra to a single
diagram.
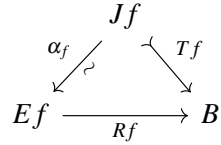
\begin{equation*}
	\begin{tikzcd}[column sep = small]
	& 
	Jf
	\arrow[ld, "\alpha_{f}"', "\sim"{rotate=45, anchor=north, pos = 0.4}]
	\arrow[rd, tail, "Tf"]
	&
	\\
	Ef
	\arrow[rr, "Rf"']
	& & 
	B
	\end{tikzcd}
\end{equation*}

In Example~\ref{example:JR-algebras}, we unpacked the definition 
of the category $Jf$ and the discrete opfibration 
$Tf$. 
We now provide an explicit characterisation of the category $Ef$
and the delta lens structure on $Rf \colon Ef \rightarrow B$.  

\begin{example}
\label{example:free-delta-lens}
The objects of $Ef$ 
are pairs $(a \in A, u \colon fa \rightarrow b \in B)$.
The morphisms are generated by pairs 
$\langle w, fw \rangle \colon (a, 1_{fa}) \rightarrow (a', 1_{fa'})$ 
and $\langle 1_{a}, v \rangle \colon (a, u) \rightarrow (a, v \circ u)$ 
for $w \in A$ and $v \in B$, respectively, as depicted below.
The identity morphisms are well-defined since $f(1_{a}) = 1_{fa}$.
As $Jf$ has the same objects as $Ef$ and consists of 
morphisms of the form $\langle 1_{a}, v \rangle$, 
the functor $\alpha_{f} \colon Jf \rightarrow Ef$ 
is identity-on-objects and faithful.  
\begin{equation}
\label{equation:Ef-generators}
    \begin{tikzcd}
    a
    \arrow[r, "w"]
    &
    a'
    \\[-2ex]
    fa
    \arrow[d, "1_{fa}"']
    \arrow[r, "f(w)"]
    &
    fa'
    \arrow[d, "1_{fa'}"]
    \\
    fa
    \arrow[r, "fw"]
    &
    fa'
    \end{tikzcd}
\qquad \qquad
	\begin{tikzcd}
    a
    \arrow[r, "1_{a}"]
    &
    a
    \\[-2ex]
    fa
    \arrow[d, "u"']
    \arrow[r, "{f(1_{a})}"]
    &
    fa
    \arrow[d, "v \, \circ \, u"]
    \\
    b
    \arrow[r, "v"]
    &
    b'
    \end{tikzcd}
\end{equation}
The functor $Rf \colon Ef \rightarrow B$ is projection in the second 
component; on the generators this is given by
$Rf\langle w, fw \rangle = fw$ and $Rf\langle 1_{a}, v \rangle = v$.
The lifting operation on $Rf$ takes an object $(a, u)$ in $Ef$ and a 
morphism $v \colon \cod(u) \rightarrow b$ in $B$ to the chosen lift 
$\langle 1_{a}, v \rangle \colon (a, u) \rightarrow (a, v \circ u)$ 
in $Ef$. 

Although, in principle, the morphisms in $Ef$ are finite sequences of 
the generators \eqref{equation:Ef-generators}, one may show that 
each morphism $(a_{1}, u_{1}) \rightarrow (a_{2}, u_{2})$ 
is actually just one of the following two kinds depicted below:
either a retraction $v$ of $u_{1}$ followed by morphism 
$w \colon a_{1} \rightarrow a_{2}$, or a morphism 
$v \colon \cod(u_{1})~\rightarrow~\cod(u_{2})$ such that 
$v \circ u_{1} = u_{2}$. 
The functor $Rf$ sends these morphisms to $u_{2} \circ fw \circ v$ 
and $v$, respectively. 
\begin{equation*}
    \begin{tikzcd}
    a_{1}
    \arrow[r, equal]
    &
    a_{1}
    \arrow[r, "w"]
    &
    a_{2}
    \arrow[r, equal]
    & 
    a_{2}
    \\[-2ex]
    fa_{1}
    \arrow[d, "u_{1}"']
    \arrow[r, equal]
    \arrow[rd, phantom, "\circlearrowright"]
    &
    fa_{1}
    \arrow[d, "1"']
    \arrow[r, "f(w)"]
    &
    fa_{2}
    \arrow[d, "1"]
    \arrow[r, equal]
    &
    fa_{2}
    \arrow[d, "u_{2}"]
    \\
    b_{1}
    \arrow[r, "v"]
    &
    fa_{1}
    \arrow[r, "fw"]
    &
    fa_{2}
    \arrow[r, "u_{2}"]
    &
    b_{2}
    \end{tikzcd}
\qquad \qquad
	\begin{tikzcd}
    a_{1}
    \arrow[r, equal]
    &
    a_{2}
    \\[-2ex]
    fa_{1}
    \arrow[d, "u_{1}"']
    \arrow[r, equal]
    \arrow[rd, phantom, "\circlearrowright"]
    &
    fa_{2}
    \arrow[d, "u_{2}"]
    \\
    b_{1}
    \arrow[r, "v"]
    &
    b_{2}
    \end{tikzcd}
\end{equation*}
\end{example}

\section{Delta lenses as the R-algebras of an algebraic weak factorisation system} 
\label{section:awfs}

In this section, 
let $(\E, \M)$ be an orthogonal factorisation system on a category $\C$
with (chosen) pushouts, 
and let $(M, \iota)$ be an idempotent comonad on 
$\C$ such that $M \colon \C \rightarrow \C$ preserves pushouts.

\subsection{Constructing the \textsc{awfs} for delta lenses}

Thus far we have constructed a functorial factorisation $(L, E, R)$ 
on $\C$ (Lemma~\ref{lemma:functorial-factorisation}), 
and extended the pointed endofunctor $(R, \eta)$ to a monad 
$(R, \eta, \mu)$ on $\C^{\mathbf{2}}$ 
(Theorem~\ref{theorem:monad-for-lenses}). 
We now show that the copointed endofunctor $(L, \epsilon)$ extends to 
a comonad $(L, \epsilon, \Delta)$, therefore completing 
the data required to describe an algebraic weak factorisation 
system on $\C$. 
For $\C = \Cat$ equipped with the comprehensive factorisation system 
and the discrete category monad, this yields an \textsc{awfs} whose
$R$-algebras are precisely delta lenses. 

First we construct the morphism $L^{2}f \colon A \rightarrow ELf$ as
on the left below. 
Using this diagram and \eqref{equation:endofunctor-R-objects}, 
it follows that 
$TLf \circ SLf = Lf \circ \iota_{A} = \alpha_{f} \circ Sf$ and 
there is solid commutative diagram as on the right below.
By the orthogonality property, there exists a unique morphism 
$\delta_{f} \colon Jf \rightarrow JLf$ as shown. 
\begin{equation}
\label{equation:constructing-comonad}
	\begin{tikzcd}
	MA
	\arrow[r, "\iota_{A}", "\sim"']
	\arrow[d, two heads, "SLf"']
	\arrow[rd, phantom, "\ulcorner", very near end]
	& 
	A
	\arrow[d, two heads, "L^{2}f"'{yshift = 3pt}]
	\arrow[dd, bend left = 50, "Lf"]
	\\
	JLf
	\arrow[r, "\alpha_{Lf}", "\sim"']
	\arrow[d, tail, "TLf"']
	&
	ELf
	\arrow[d, "RLf"']
	\\
	Ef
	\arrow[r, equal]
	&
	Ef
	\end{tikzcd}
\qquad \qquad \qquad
	\begin{tikzcd}
	MA
	\arrow[r, two heads, "SLf"]
	\arrow[d, two heads, "Sf"']
	&
	JLf
	\arrow[d, tail, "TLf"]
	\\
	Jf
	\arrow[r, "\alpha_{f}", "\sim"']
	\arrow[ru, dashed, two heads, "\delta_{f}"]
	&
	Ef
	\end{tikzcd}
\end{equation}

Using the diagrams \eqref{equation:constructing-comonad} and the 
universal property of the pushout, we obtain the component 
$\Delta_{f}$ of the comultiplication $\Delta \colon L \Rightarrow L^{2}$
at $f$ as depicted below.
For each morphism $\langle h, k \rangle \colon f \rightarrow g$ in 
$\C^{\mathbf{2}}$, 
we may show that 
$\Delta_{g} \circ E\langle h, k \rangle =
E\langle h, E\langle h, k \rangle \rangle \circ \Delta_{f}$, 
providing us with a well-defined transformation 
$\Delta \colon L \Rightarrow L^{2}$.
\begin{equation*}
    \begin{tikzcd}
    MA
	\arrow[r, "\iota_{A}", "\sim"']
	\arrow[d, two heads, "Sf"']
	\arrow[rd, phantom, "\ulcorner", very near end]
	& 
	A
	\arrow[d, two heads, "Lf"]
	\arrow[rdd, bend left, two heads, "L^{2}f"]
	&
	\\
	Jf
	\arrow[r, "\alpha_{f}", "\sim"']
	\arrow[rd, two heads, "\delta_{f}"']
	&
	Ef
	\arrow[rd, dashed, two heads, "\Delta_{f}"]
	&
	\\
	&
	JLf
	\arrow[r, "\alpha_{Lf}", "\sim"']
	&
	ELf
    \end{tikzcd}
\qquad = \qquad
	 \begin{tikzcd}
    MA
	\arrow[r, "\iota_{A}", "\sim"']
	\arrow[d, two heads, "Sf"']
	\arrow[rrdd, phantom, "\ulcorner", very near end]
	\arrow[rdd, two heads, bend left, "SLf"]
	& 
	A
	\arrow[rdd, bend left, two heads, "L^{2}f"]
	&
	\\
	Jf
	\arrow[rd, two heads, "\delta_{f}"']
	&
	&
	\\
	&
	JLf
	\arrow[r, "\alpha_{Lf}", "\sim"']
	&
	ELf
    \end{tikzcd}
\end{equation*}
Showing that the diagrams below commute, and thus establishing that the 
comultiplication $\Delta$ 
is counital and coassociative, is a straightforward
application of definitions and the universal property of a pushout.
\begin{equation*}
	\begin{tikzcd}[column sep = large]
	& 
	Ef
	\arrow[ld, equal]
	\arrow[rd, equal]
	\arrow[d, "\Delta_{f}"]
	&
	\\
	Ef
	&
	ELf
	\arrow[l, "RLf"]
	\arrow[r, "{\langle 1_{A},\, Rf\rangle}"']
	& 
	Ef
	\end{tikzcd}
\qquad \qquad \qquad
	\begin{tikzcd}[column sep = large]
	Ef
	\arrow[r, "\Delta_{f}"]
	\arrow[d, "\Delta_{f}"']
	& 
	ELf
	\arrow[d, "\Delta_{Lf}"]
	\\
	ELf
	\arrow[r, "{E\langle 1_{A},\, \Delta_{f}\rangle}"']
	& 
	EL^{2}f
	\end{tikzcd}
\end{equation*}

\begin{proposition}
\label{proposition:comonad-L}
The triple $(L, \epsilon, \Delta)$ is a comonad on $\C^{\mathbf{2}}$. 
\end{proposition}

\begin{theorem}
\label{theorem:awfs}
The pair $(L, R)$ is an algebraic weak factorisation system on $\C$. 
\end{theorem}
\begin{proof}
The data of the algebraic weak factorisation system 
follows from Lemma~\ref{lemma:functorial-factorisation}, 
Theorem~\ref{theorem:monad-for-lenses}, and 
Proposition~\ref{proposition:comonad-L}. 
Checking that there is a distributive law 
$\lambda \colon LR \Rightarrow RL$ of the comonad $L$
over the monad $R$ with components 
$\lambda_{f} = \langle \Delta_{f}, \mu_{f} \rangle$ involves  
routine diagram-chasing and applying universal properties.
\end{proof}

\subsection{Coalgebras and lifting}

A \emph{coalgebra} $(f, q)$ for the comonad $(L, \epsilon, \Delta)$ 
consists of a pair of morphisms $f \colon A \rightarrow B$ and 
$q \colon B \rightarrow Ef$ such that the following diagrams commute: 
\begin{equation*}
	\begin{tikzcd}
	A
	\arrow[d, "f"']
	\arrow[r, two heads, "Lf"]
	&
	Ef
	\arrow[d, "Rf"]
	\\
	B
	\arrow[r, equal]
	\arrow[ru, "q"]
	&
	B
	\end{tikzcd}
\qquad \qquad
	\begin{tikzcd}[column sep = large]
	B
	\arrow[r, "q"]
	\arrow[d, "q"']
	&
	Ef
	\arrow[d, "\Delta_{f}"]
	\\
	Ef
	\arrow[r, "{E\langle 1_{A},\, q \rangle}"']
	&
	ELf
	\end{tikzcd}
\end{equation*}

\begin{remark}
In contrast to the algebras for the monad $(R, \eta, \mu)$, 
the coalgebras above cannot be easily simplified since $q$ is a morphism 
\emph{into} a pushout. 
For $\C = \Cat$, one may show that for a functor $f$ to admit a 
coalgebra structure, it must be a left-adjoint-right-inverse 
(\textsc{lari})
and is therefore also injective-on-objects and fully faithful. 
A complete characterisation of the $L$-coalgebras is left for future
work. 
\end{remark}

We now provide a simple diagrammatic proof that delta lenses, 
in the form of Proposition~\ref{proposition:delta-lens-representation}
rather than as $R$-algebras, lift against $L$-coalgebras. 
Consider a morphism $\langle h, k \rangle \colon f \rightarrow g$ 
such that $(f, q)$ is an $L$-coalgebra and $(g, \psi)$ is a delta lens.
Since $\psi$ is bijective-on-objects, $\psi_{0}$ is invertible, 
and there is a morphism 
$\iota_{\Lambda} \circ \psi_{0}^{-1} \circ h_{0} \colon A_{0} \rightarrow \Lambda(g, \psi)$ making the diagram, depicted below, commute. 
Then by the orthogonality property, there exists a unique morphism 
$\ell \colon Jf \rightarrow \Lambda(g, \psi)$ such that 
$\ell \circ Sf = \iota_{\Lambda} \circ \psi_{0}^{-1} \circ h_{0}$
and $g \circ \psi \circ \ell = k \circ Rf \circ \alpha_{f}$. 
Finally, by the universal property of the pushout, there exists a
unique morphism $[\psi \circ \ell, h] \colon Ef \rightarrow C$. 
Thus, there is a specified morphism 
$q \circ [\psi \circ \ell, h] \colon B \rightarrow C$ as on the left 
below.
\begin{equation*}
	\begin{tikzcd}[column sep = large, row sep = large]
	A
	\arrow[r, "h"]
	\arrow[d, "f"']
	&
	C
	\arrow[d, "g"]
	\\
	B
	\arrow[r, "k"']
	\arrow[ru, "{q \,\circ\, [\psi \,\circ\, \ell,\, h]}" description]
	&
	D
	\end{tikzcd}
\qquad \qquad \qquad 
	\begin{tikzcd}
	& 
	& 
	\Lambda(g, \psi)
	\arrow[d, "\psi", "\sim"{rotate=90, anchor = south}]
	\arrow[dd, bend left = 50, tail, "g \,\circ\, \psi"]
	\\[-1ex]
	A_{0}
	\arrow[rru, bend left = 20, dashed,
	"\iota_{\Lambda} \,\circ\, \psi_{0}^{-1} \,\circ\, h_{0}"{pos = 0.2}]
	\arrow[r, "\iota_{A}", "\sim"']
	\arrow[d, two heads, "Sf"']
	\arrow[rd, phantom, "\ulcorner", very near end]
	&
	A
	\arrow[r, "h"]
	\arrow[d, two heads, "Lf"]
	&
	C
	\arrow[d, "g"]
	\\[-1ex]
	Jf
	\arrow[r, "\alpha_{f}", "\sim"']
	&
	Ef
	\arrow[ru, dashed, "{\psi \circ l,\, h}" description]
	\arrow[r, "k \,\circ\, Rf"']
	&
	D
	\end{tikzcd}
\end{equation*}
Therefore we have shown that delta lenses lift against functors with 
the structure of a $L$-coalgebra, which is stronger than one would 
expect from their simple axiomatic definition.   
It also demonstrates how the notion of lifting is intrinsic to delta
lenses as the $R$-algebras of an \textsc{awfs}. 
The sequential composition of delta lenses as $R$-algebras may also be 
defined from this notion of lifting against $L$-coalgebras, 
providing further clarification of this essential operation. 

\section{Concluding remarks and future work}
\label{section:conclusion}

In this paper, we have shown that delta lenses are algebras for a 
monad $(R, \eta, \mu)$, and that this monad arises from an algebraic
weak factorisation system on $\Cat$. 
Moreover, we have shown that this \textsc{awfs} exists on any suitable
category  equipped with an orthogonal factorisation system and 
an idempotent comonad which preserves pushouts. 
These results generalise immediately to \emph{internal lenses}
\cite{Cla20, TAC} using the internal comprehensive factorisation system
\cite{SV10}, however an analogous result for \emph{enriched lenses} \cite{CDi22} or \emph{weighted lenses} \cite{Per21} 
is unknown.
There are many avenues for future work.
One example is the relationship between the 
\emph{proxy pullbacks} \cite{DiM21} of delta lenses and the canonical
pullback of $R$-algebras~\cite{BG16}.
Another is the connection between spans of delta lenses \cite{Cla21} 
and the categories of weak maps for an \textsc{awfs}~\cite{BG16II}.  
The \emph{double category of delta lenses} \cite{Cla22}, 
which is naturally induced by the \textsc{awfs},
provides a rich setting studying the properties
of delta lenses previously considered in a $1$-categorical setting 
\cite{CCJSWZ22, DiM22}. 

\textbf{Acknowledgements.}
This research first appeared in Chapter~6 of my PhD thesis 
\cite{Cla22}, and I would like to thank my supervisor, 
Michael Johnson, for his helpful feedback during my PhD studies.
I am also very grateful to Richard Garner who first sketched the 
construction of the free delta lens, and also suggested the approach 
using \textsc{awfs}, after I presented this work at the Australian 
Category Seminar. 
I would also like to thank 
John Bourke, Matthew Di Meglio, Tim Hosgood, and Noam Zeilberger for
sharing useful insights which contributed to the development of 
this research. 

\bibliographystyle{eptcs}
\bibliography{ACT2023.bib}

\begin{thebibliography}{10}
\providecommand{\bibitemdeclare}[2]{}
\providecommand{\surnamestart}{}
\providecommand{\surnameend}{}
\providecommand{\urlprefix}{Available at }
\providecommand{\url}[1]{\texttt{#1}}
\providecommand{\href}[2]{\texttt{#2}}
\providecommand{\urlalt}[2]{\href{#1}{#2}}
\providecommand{\doi}[1]{doi:\urlalt{https://doi.org/#1}{#1}}
\providecommand{\eprint}[1]{arXiv:\urlalt{https://arxiv.org/abs/#1}{#1}}
\providecommand{\bibinfo}[2]{#2}

\bibitemdeclare{inproceedings}{ACGMS18}
\bibitem{ACGMS18}
\bibinfo{author}{Faris \surnamestart Abou-Saleh\surnameend},
  \bibinfo{author}{James \surnamestart Cheney\surnameend},
  \bibinfo{author}{Jeremy \surnamestart Gibbons\surnameend},
  \bibinfo{author}{James \surnamestart McKinna\surnameend} \&
  \bibinfo{author}{Perdita \surnamestart Stevens\surnameend}
  (\bibinfo{year}{2018}): \emph{\bibinfo{title}{Introduction to Bidirectional
  Transformations}}.
\newblock In \bibinfo{editor}{Jeremy \surnamestart Gibbons\surnameend} \&
  \bibinfo{editor}{Perdita \surnamestart Stevens\surnameend}, editors:
  {\slshape \bibinfo{booktitle}{Bidirectional Transformations}}, {\slshape
  \bibinfo{series}{Lecture Notes in Computer Science}} \bibinfo{volume}{9715},
  pp. \bibinfo{pages}{1--28}, \doi{10.1007/978-3-319-79108-1_1}.

\bibitemdeclare{phdthesis}{Agu97}
\bibitem{Agu97}
\bibinfo{author}{Marcelo \surnamestart Aguiar\surnameend}
  (\bibinfo{year}{1997}): \emph{\bibinfo{title}{Internal categories and quantum
  groups}}.
\newblock Ph.D. thesis, \bibinfo{school}{Cornell University}.
\newblock \urlprefix\url{http://pi.math.cornell.edu/~maguiar/thesis2.pdf}.

\bibitemdeclare{inproceedings}{AU17}
\bibitem{AU17}
\bibinfo{author}{Danel \surnamestart Ahman\surnameend} \&
  \bibinfo{author}{Tarmo \surnamestart Uustalu\surnameend}
  (\bibinfo{year}{2017}): \emph{\bibinfo{title}{Taking Updates Seriously}}.
\newblock In \bibinfo{editor}{Romina \surnamestart Eramo\surnameend} \&
  \bibinfo{editor}{Michael \surnamestart Johnson\surnameend}, editors:
  {\slshape \bibinfo{booktitle}{Proceedings of the 6th International Workshop
  on Bidirectional Transformations}}, {\slshape \bibinfo{series}{CEUR Workshop
  Proceedings}} \bibinfo{volume}{1827}, pp. \bibinfo{pages}{59--73}.
\newblock \urlprefix\url{https://ceur-ws.org/Vol-1827/paper11.pdf}.

\bibitemdeclare{article}{BG16}
\bibitem{BG16}
\bibinfo{author}{John \surnamestart Bourke\surnameend} \&
  \bibinfo{author}{Richard \surnamestart Garner\surnameend}
  (\bibinfo{year}{2016}): \emph{\bibinfo{title}{Algebraic weak factorisation
  systems I: Accessible AWFS}}.
\newblock {\slshape \bibinfo{journal}{Journal of Pure and Applied Algebra}}
  \bibinfo{volume}{220}(\bibinfo{number}{1}), \doi{10.1016/j.jpaa.2015.06.002}.

\bibitemdeclare{article}{BG16II}
\bibitem{BG16II}
\bibinfo{author}{John \surnamestart Bourke\surnameend} \&
  \bibinfo{author}{Richard \surnamestart Garner\surnameend}
  (\bibinfo{year}{2016}): \emph{\bibinfo{title}{Algebraic weak factorisation
  systems II: Categories of weak maps}}.
\newblock {\slshape \bibinfo{journal}{Journal of Pure and Applied Algebra}}
  \bibinfo{volume}{220}(\bibinfo{number}{1}), \doi{10.1016/j.jpaa.2015.06.003}.

\bibitemdeclare{inproceedings}{CCJSWZ22}
\bibitem{CCJSWZ22}
\bibinfo{author}{Emma \surnamestart Chollet\surnameend}, \bibinfo{author}{Bryce
  \surnamestart Clarke\surnameend}, \bibinfo{author}{Michael \surnamestart
  Johnson\surnameend}, \bibinfo{author}{Maurine \surnamestart
  Songa\surnameend}, \bibinfo{author}{Vincent \surnamestart Wang\surnameend} \&
  \bibinfo{author}{Gioele \surnamestart Zardini\surnameend}
  (\bibinfo{year}{2022}): \emph{\bibinfo{title}{Limits and Colimits in a
  Category of Lenses}}.
\newblock In \bibinfo{editor}{Kohei \surnamestart Kishida\surnameend}, editor:
  {\slshape \bibinfo{booktitle}{Proceedings of the Fourth International
  Conference on Applied Category Theory}}, {\slshape
  \bibinfo{series}{Electronic Proceedings in Theoretical Computer Science}}
  \bibinfo{volume}{372}, pp. \bibinfo{pages}{164--177},
  \doi{10.4204/EPTCS.372.12}.

\bibitemdeclare{inproceedings}{Cla20}
\bibitem{Cla20}
\bibinfo{author}{Bryce \surnamestart Clarke\surnameend} (\bibinfo{year}{2020}):
  \emph{\bibinfo{title}{Internal lenses as functors and cofunctors}}.
\newblock In \bibinfo{editor}{John \surnamestart Baez\surnameend} \&
  \bibinfo{editor}{Bob \surnamestart Coecke\surnameend}, editors: {\slshape
  \bibinfo{booktitle}{Proceedings Applied Category Theory 2019}}, {\slshape
  \bibinfo{series}{Electronic Proceedings in Theoretical Computer Science}}
  \bibinfo{volume}{323}, pp. \bibinfo{pages}{183--195},
  \doi{10.4204/EPTCS.323.13}.

\bibitemdeclare{article}{TAC}
\bibitem{TAC}
\bibinfo{author}{Bryce \surnamestart Clarke\surnameend} (\bibinfo{year}{2020}):
  \emph{\bibinfo{title}{Internal split opfibrations and cofunctors}}.
\newblock {\slshape \bibinfo{journal}{Theory and Applications of Categories}}
  \bibinfo{volume}{35}(\bibinfo{number}{44}).
\newblock
  \urlprefix\url{http://www.tac.mta.ca/tac/volumes/35/44/35-44abs.html}.

\bibitemdeclare{inproceedings}{Cla21}
\bibitem{Cla21}
\bibinfo{author}{Bryce \surnamestart Clarke\surnameend} (\bibinfo{year}{2021}):
  \emph{\bibinfo{title}{Delta Lenses as Coalgebras for a Comonad}}.
\newblock In \bibinfo{editor}{Leen \surnamestart Lambers\surnameend} \&
  \bibinfo{editor}{Meng \surnamestart Wang\surnameend}, editors: {\slshape
  \bibinfo{booktitle}{9th International Workshop on Bidirectional
  Transformations}}, {\slshape \bibinfo{series}{CEUR Workshop Proceedings}}
  \bibinfo{volume}{2999}, pp. \bibinfo{pages}{18--27}.
\newblock \urlprefix\url{https://ceur-ws.org/Vol-2999/bxpaper2.pdf}.

\bibitemdeclare{inproceedings}{ACT2020}
\bibitem{ACT2020}
\bibinfo{author}{Bryce \surnamestart Clarke\surnameend} (\bibinfo{year}{2021}):
  \emph{\bibinfo{title}{A diagrammatic approach to symmetric lenses}}.
\newblock In \bibinfo{editor}{David~I. \surnamestart Spivak\surnameend} \&
  \bibinfo{editor}{Jamie \surnamestart Vicary\surnameend}, editors: {\slshape
  \bibinfo{booktitle}{Proceedings of the 3rd Annual International Applied
  Category Theory Conference 2020}}, {\slshape \bibinfo{series}{Electronic
  Proceedings in Theoretical Computer Science}} \bibinfo{volume}{333}, pp.
  \bibinfo{pages}{79--91}, \doi{10.4204/EPTCS.333.6}.

\bibitemdeclare{phdthesis}{Cla22}
\bibitem{Cla22}
\bibinfo{author}{Bryce \surnamestart Clarke\surnameend} (\bibinfo{year}{2022}):
  \emph{\bibinfo{title}{The double category of lenses}}.
\newblock Ph.D. thesis, \bibinfo{school}{Macquarie University},
  \doi{10.25949/22045073.v1}.

\bibitemdeclare{misc}{CDi22}
\bibitem{CDi22}
\bibinfo{author}{Bryce \surnamestart Clarke\surnameend} \&
  \bibinfo{author}{Matthew \surnamestart {Di Meglio}\surnameend}
  (\bibinfo{year}{2022}): \emph{\bibinfo{title}{An introduction to enriched
  cofunctors}}.
\newblock \eprint{2209.01144}.

\bibitemdeclare{inproceedings}{CFHLST09}
\bibitem{CFHLST09}
\bibinfo{author}{Krzysztof \surnamestart Czarnecki\surnameend},
  \bibinfo{author}{J.~Nathan \surnamestart Foster\surnameend},
  \bibinfo{author}{Zhenjiang \surnamestart Hu\surnameend},
  \bibinfo{author}{Ralf \surnamestart Lämmel\surnameend},
  \bibinfo{author}{Andy \surnamestart Schürr\surnameend} \&
  \bibinfo{author}{James~F. \surnamestart Terwilliger\surnameend}
  (\bibinfo{year}{2009}): \emph{\bibinfo{title}{Bidirectional Transformations:
  A Cross-Discipline Perspective}}.
\newblock In \bibinfo{editor}{Richard~F. \surnamestart Paige\surnameend},
  editor: {\slshape \bibinfo{booktitle}{Theory and Practice of Model
  Transformations}}, {\slshape \bibinfo{series}{Lecture Notes in Computer
  Science}} \bibinfo{volume}{5563}, pp. \bibinfo{pages}{260--283},
  \doi{10.1007/978-3-642-02408-5_19}.

\bibitemdeclare{mastersthesis}{DiM21}
\bibitem{DiM21}
\bibinfo{author}{Matthew \surnamestart {Di Meglio}\surnameend}
  (\bibinfo{year}{2021}): \emph{\bibinfo{title}{The category of asymmetric
  lenses and its proxy pullbacks}}.
\newblock Master's thesis, \bibinfo{school}{Macquarie University},
  \doi{10.25949/20236449.v1}.

\bibitemdeclare{inproceedings}{DiM22}
\bibitem{DiM22}
\bibinfo{author}{Matthew \surnamestart Di~Meglio\surnameend}
  (\bibinfo{year}{2022}): \emph{\bibinfo{title}{Coequalisers under the Lens}}.
\newblock In \bibinfo{editor}{Kohei \surnamestart Kishida\surnameend}, editor:
  {\slshape \bibinfo{booktitle}{Proceedings of the Fourth International
  Conference on Applied Category Theory}}, {\slshape
  \bibinfo{series}{Electronic Proceedings in Theoretical Computer Science}}
  \bibinfo{volume}{372}, pp. \bibinfo{pages}{149--163},
  \doi{10.4204/EPTCS.372.11}.

\bibitemdeclare{inproceedings}{DM12}
\bibitem{DM12}
\bibinfo{author}{Zinovy \surnamestart Diskin\surnameend} \&
  \bibinfo{author}{Tom \surnamestart Maibaum\surnameend}
  (\bibinfo{year}{2012}): \emph{\bibinfo{title}{Category Theory and
  Model-Driven Engineering: From Formal Semantics to Design Patterns and
  Beyond}}.
\newblock In \bibinfo{editor}{Ulrike \surnamestart Golas\surnameend} \&
  \bibinfo{editor}{Thomas \surnamestart Soboll\surnameend}, editors: {\slshape
  \bibinfo{booktitle}{Proceedings Seventh ACCAT Workshop on Applied and
  Computational Category Theory}}, {\slshape \bibinfo{series}{Electronic
  Proceedings in Theoretical Computer Science}}~\bibinfo{volume}{93}, pp.
  \bibinfo{pages}{1--21}, \doi{10.4204/EPTCS.93.1}.

\bibitemdeclare{article}{DXC11}
\bibitem{DXC11}
\bibinfo{author}{Zinovy \surnamestart Diskin\surnameend},
  \bibinfo{author}{Yingfei \surnamestart Xiong\surnameend} \&
  \bibinfo{author}{Krzysztof \surnamestart Czarnecki\surnameend}
  (\bibinfo{year}{2011}): \emph{\bibinfo{title}{From State- to Delta-Based
  Bidirectional Model Transformations: the Asymmetric Case}}.
\newblock {\slshape \bibinfo{journal}{Journal of Object Technology}}
  \bibinfo{volume}{10}(\bibinfo{number}{6}), \doi{10.5381/jot.2011.10.1.a6}.

\bibitemdeclare{book}{DHKS04}
\bibitem{DHKS04}
\bibinfo{author}{William~G. \surnamestart Dwyer\surnameend},
  \bibinfo{author}{Philip~S. \surnamestart Hirschhorn\surnameend},
  \bibinfo{author}{Daniel~M. \surnamestart Kan\surnameend} \&
  \bibinfo{author}{Jeffrey~H. \surnamestart Smith\surnameend}
  (\bibinfo{year}{2004}): \emph{\bibinfo{title}{Homotopy Limit Functors on
  Model Categories and Homotopical Categories}}.
\newblock {\slshape \bibinfo{series}{Mathematical Surveys and Monographs}}
  \bibinfo{volume}{113}, \bibinfo{publisher}{American Mathematical Society},
  \doi{10.1090/surv/113}.

\bibitemdeclare{article}{FGMPS07}
\bibitem{FGMPS07}
\bibinfo{author}{J.~Nathan \surnamestart Foster\surnameend},
  \bibinfo{author}{Michael~B. \surnamestart Greenwald\surnameend},
  \bibinfo{author}{Jonathan~T. \surnamestart Moore\surnameend},
  \bibinfo{author}{Benjamin~C. \surnamestart Pierce\surnameend} \&
  \bibinfo{author}{Alan \surnamestart Schmitt\surnameend}
  (\bibinfo{year}{2007}): \emph{\bibinfo{title}{Combinators for bidirectional
  tree transformations: A linguistic approach to the view-update problem}}.
\newblock {\slshape \bibinfo{journal}{ACM Transactions on Programming Languages
  and Systems}} \bibinfo{volume}{29}(\bibinfo{number}{3}),
  \doi{10.1145/1232420.1232424}.

\bibitemdeclare{article}{FK72}
\bibitem{FK72}
\bibinfo{author}{P.J. \surnamestart Freyd\surnameend} \& \bibinfo{author}{G.M.
  \surnamestart Kelly\surnameend} (\bibinfo{year}{1972}):
  \emph{\bibinfo{title}{Categories of continuous functors I}}.
\newblock {\slshape \bibinfo{journal}{Journal of Pure and Applied Algebra}}
  \bibinfo{volume}{2}(\bibinfo{number}{3}), \doi{10.1016/0022-4049(72)90001-1}.

\bibitemdeclare{article}{GT06}
\bibitem{GT06}
\bibinfo{author}{Marco \surnamestart Grandis\surnameend} \&
  \bibinfo{author}{Walter \surnamestart Tholen\surnameend}
  (\bibinfo{year}{2006}): \emph{\bibinfo{title}{Natural weak factorization
  systems}}.
\newblock {\slshape \bibinfo{journal}{Archivum mathematicum}}
  \bibinfo{volume}{42}(\bibinfo{number}{4}).
\newblock \urlprefix\url{https://www.emis.de/journals/AM/06-4/tholen.pdf}.

\bibitemdeclare{article}{HM93}
\bibitem{HM93}
\bibinfo{author}{Philip~J. \surnamestart Higgins\surnameend} \&
  \bibinfo{author}{Kirill C.~H. \surnamestart Mackenzie\surnameend}
  (\bibinfo{year}{1993}): \emph{\bibinfo{title}{Duality for base-changing
  morphisms of vector bundles, modules, Lie algebroids and Poisson
  structures}}.
\newblock {\slshape \bibinfo{journal}{Mathematical Proceedings of the Cambridge
  Philosophical Society}} \bibinfo{volume}{114}(\bibinfo{number}{3}),
  \doi{10.1017/S0305004100071760}.

\bibitemdeclare{inproceedings}{JR13}
\bibitem{JR13}
\bibinfo{author}{Michael \surnamestart Johnson\surnameend} \&
  \bibinfo{author}{Robert \surnamestart Rosebrugh\surnameend}
  (\bibinfo{year}{2013}): \emph{\bibinfo{title}{Delta Lenses and
  Opfibrations}}.
\newblock In \bibinfo{editor}{Perdita \surnamestart Stevens\surnameend} \&
  \bibinfo{editor}{James~F. \surnamestart Terwilliger\surnameend}, editors:
  {\slshape \bibinfo{booktitle}{Proceedings of the Second International
  Workshop on Bidirectional Transformations}}, {\slshape
  \bibinfo{series}{Electronic Communications of the
  EASST}}~\bibinfo{volume}{57}, pp. \bibinfo{pages}{1--18},
  \doi{10.14279/tuj.eceasst.57.875}.

\bibitemdeclare{inproceedings}{JR16}
\bibitem{JR16}
\bibinfo{author}{Michael \surnamestart Johnson\surnameend} \&
  \bibinfo{author}{Robert \surnamestart Rosebrugh\surnameend}
  (\bibinfo{year}{2016}): \emph{\bibinfo{title}{Unifying Set-Based, Delta-Based
  and Edit-Based Lenses}}.
\newblock In \bibinfo{editor}{Anthony \surnamestart Anjorin\surnameend} \&
  \bibinfo{editor}{Jeremy \surnamestart Gibbons\surnameend}, editors: {\slshape
  \bibinfo{booktitle}{Proceedings of the Fifth International Workshop on
  Bidirectional Transformations}}, {\slshape \bibinfo{series}{CEUR Workshop
  Proceedings}} \bibinfo{volume}{1571}, pp. \bibinfo{pages}{1--13}.
\newblock \urlprefix\url{https://ceur-ws.org/Vol-1571/paper_13.pdf}.

\bibitemdeclare{article}{JRW12}
\bibitem{JRW12}
\bibinfo{author}{Michael \surnamestart Johnson\surnameend},
  \bibinfo{author}{Robert \surnamestart Rosebrugh\surnameend} \&
  \bibinfo{author}{R.~J. \surnamestart Wood\surnameend} (\bibinfo{year}{2012}):
  \emph{\bibinfo{title}{Lenses, fibrations and universal translations}}.
\newblock {\slshape \bibinfo{journal}{Mathematical Structures in Computer
  Science}} \bibinfo{volume}{22}(\bibinfo{number}{1}),
  \doi{10.1017/S0960129511000442}.

\bibitemdeclare{article}{JRW10}
\bibitem{JRW10}
\bibinfo{author}{Michael \surnamestart Johnson\surnameend},
  \bibinfo{author}{Robert \surnamestart Rosebrugh\surnameend} \&
  \bibinfo{author}{Richard \surnamestart Wood\surnameend}
  (\bibinfo{year}{2010}): \emph{\bibinfo{title}{Algebras and Update
  Strategies}}.
\newblock {\slshape \bibinfo{journal}{Journal of Universal Computer Science}}
  \bibinfo{volume}{16}(\bibinfo{number}{5}), \doi{10.3217/jucs-016-05-0729}.

\bibitemdeclare{misc}{Per21}
\bibitem{Per21}
\bibinfo{author}{Paolo \surnamestart Perrone\surnameend}
  (\bibinfo{year}{2021}): \emph{\bibinfo{title}{Lifting couplings in
  Wasserstein spaces}}.
\newblock \eprint{2110.06591}.

\bibitemdeclare{article}{Sil15}
\bibitem{Sil15}
\bibinfo{author}{Alberto \surnamestart {Rodrigues da Silva}\surnameend}
  (\bibinfo{year}{2015}): \emph{\bibinfo{title}{Model-driven engineering: A
  survey supported by the unified conceptual model}}.
\newblock {\slshape \bibinfo{journal}{Computer Languages, Systems \&
  Structures}} \bibinfo{volume}{43}, \doi{10.1016/j.cl.2015.06.001}.

\bibitemdeclare{inproceedings}{Str74}
\bibitem{Str74}
\bibinfo{author}{Ross \surnamestart Street\surnameend} (\bibinfo{year}{1974}):
  \emph{\bibinfo{title}{Fibrations and Yoneda's lemma in a 2-category}}.
\newblock In \bibinfo{editor}{G.~M. \surnamestart Kelly\surnameend}, editor:
  {\slshape \bibinfo{booktitle}{Category Seminar}}, {\slshape
  \bibinfo{series}{Lecture Notes in Mathematics}} \bibinfo{volume}{420}, pp.
  \bibinfo{pages}{104--133}, \doi{10.1007/BFb0063102}.

\bibitemdeclare{article}{SV10}
\bibitem{SV10}
\bibinfo{author}{Ross \surnamestart Street\surnameend} \&
  \bibinfo{author}{Dominic \surnamestart Verity\surnameend}
  (\bibinfo{year}{2010}): \emph{\bibinfo{title}{The comprehensive factorization
  and torsors}}.
\newblock {\slshape \bibinfo{journal}{Theory and Applications of Categories}}
  \bibinfo{volume}{23}(\bibinfo{number}{3}).
\newblock \urlprefix\url{http://www.tac.mta.ca/tac/volumes/23/3/23-03abs.html}.

\bibitemdeclare{article}{SW73}
\bibitem{SW73}
\bibinfo{author}{Ross \surnamestart Street\surnameend} \&
  \bibinfo{author}{R.~F.~C. \surnamestart Walters\surnameend}
  (\bibinfo{year}{1973}): \emph{\bibinfo{title}{The comprehensive factorization
  of a functor}}.
\newblock {\slshape \bibinfo{journal}{Bulletin of the American Mathematical
  Society}} \bibinfo{volume}{79}(\bibinfo{number}{5}),
  \doi{10.1090/S0002-9904-1973-13268-9}.

\end{thebibliography}

\appendix
\section{Appendix}
\label{section:appendix}

In this section, we provide a proof of  
Theorem~\ref{theorem:JR-isomorphism}. 
The correspondence between \JR-algebras and delta lenses 
was first shown by Johnson and Rosebrugh \cite[Proposition~3]{JR13};
we reprove this correspondence in our notation, and extend it 
to an isomorphism of categories. 
We refer the reader to Example~\ref{example:JR-algebras} for an 
explicit description of the category $Jf$ and the functor 
$Tf \colon Jf \rightarrow B$. 

\begin{theorem}
If $\C = \Cat$ equipped with the discrete category comonad and the 
comprehensive factorisation system, then there is an isomorphism 
of categories $\Lens \cong \JRAlg(T, \nu)$. 
\end{theorem}

We prove this theorem in two parts: first defining the 
functor $\Lens \rightarrow \JRAlg(T, \nu)$, then defining the 
functor $\JRAlg(T, \nu) \rightarrow \Lens$ and showing 
that they are mutually inverse. 

\begin{proof}
We begin by constructing a functor $\Lens \rightarrow \JRAlg(T, \nu)$. 

Given a delta lens $(f, \varphi) \colon A \rightarrow B$ as in 
Definition~\ref{definition:delta-lens}, 
we define a functor $p \colon Jf \rightarrow B$ whose assignment on 
morphisms $\langle 1_{a}, v \rangle \colon (a, u_{1}) \rightarrow (a, u_{2})$ is given below, where we recall that $p(a, u) = \cod(\varphi(a, u))$.  
\begin{equation}
\label{equation:lens-to-algebra}
	\begin{tikzcd}
    a
    \arrow[r, equal]
    &
    a
    \\[-2ex]
    fa
    \arrow[d, "u_{1}"']
    \arrow[r, equal]
    \arrow[rd, phantom, "\circlearrowright"]
    &
    fa
    \arrow[d, "u_{2}"]
    \\
    b_{1}
    \arrow[r, "v"]
    &
    b_{2}
    \end{tikzcd}
\qquad \longmapsto \qquad
	\begin{tikzcd}[column sep = huge]
	p(a, u_{1})
	\arrow[r, "{\varphi(p(a,\, u_{1}),\, v)}"]
	&
	p(a, u_{2})
	\end{tikzcd}
\end{equation}
This functor preserves identities and composition by the axioms 
\ref{(L2)} and \ref{(L3)} of a delta lens, respectively.
Moreover, the equation $f \circ p = Tf$ from the left diagram of 
\eqref{equation:semi-monad-algebra} is satisfied by axiom \ref{(L1)}.
The equation $p \circ Sf = \iota_{A}$ from the 
diagram \eqref{equation:JR-algebra} also holds since 
$Sf(a) = (a, 1_{fa})$ and $p(a, 1_{fa}) = a$ by axiom \ref{(L2)}.

To verify the remaining condition for a \JR-algebra given by the 
right diagram of \eqref{equation:semi-monad-algebra}, we first 
describe the category $JTf$ and the functors 
$\nu_{f}, \langle p, 1_{B} \rangle \colon JTf \rightarrow Jf$.

The category $JTf$ has objects given by triples 
$(a \in A, u \colon fa \rightarrow b, u' \colon b \rightarrow b')$ 
and morphisms given by triples $\langle 1_{a}, 1_{b}, v \rangle$ as 
depicted below. 
The functor $\nu_{f}$ has an assignment on objects 
$(a, u, u') \mapsto (a, u' \circ~u)$ 
and an assignment on morphisms 
$\langle 1_{a}, 1_{b}, v \rangle \mapsto \langle 1_{a}, v \rangle$,
while the functor $\langle p, 1_{B} \rangle$ has corresponding assignments on objects and morphisms given by 
$(a, u, u') \mapsto (p(a, u), u')$ and
$\langle 1_{a}, 1_{b}, v \rangle \mapsto \langle 1_{p(a, u)}, v \rangle$
which are well-defined by \ref{(L1)}. 
The equation $p \circ \mu_{f} = p \circ \langle p, 1_{B}\rangle$ holds 
since $p(a, u' \circ u) = p(p(a, u), u')$ by axiom \ref{(L3)}.
Therefore, we have a \JR-algebra $(f, p)$ and the functor 
$\Lens \rightarrow \JRAlg(T, \nu)$ is well-defined on objects.
\begin{equation}
	\begin{tikzcd}
    a
    \arrow[r, equal]
    &
    a
    \\[-2ex]
    fa
    \arrow[d, "u"']
    \arrow[r, equal]
    &
    fa
    \arrow[d, "u"]
    \\
    b
    \arrow[r, equal]
    \arrow[d, "u_{1}'"']
    \arrow[rd, phantom, "\circlearrowright"]
	&
	b
	\arrow[d, "u'_{2}"]
	\\
	b'_{1}
    \arrow[r, "v"]
    &
    b'_{2}
    \end{tikzcd}
\end{equation}

Consider a pair of delta lenses $(f, \varphi) \colon A \rightarrow B$
and $(g, \psi) \colon C \rightarrow D$ with corresponding 
\JR-algebras $(f, p)$ and $(g, q)$, respectively.
Given a morphism of delta lenses $\langle h, k \rangle \colon (f, \varphi) \rightarrow (g, \psi)$, we want to show that there is a 
morphism of \JR-algebras 
$\langle h, k \rangle \colon (f, p) \rightarrow (g, q)$. 
First note that the functor 
$J\langle h, k \rangle \colon Jf \rightarrow Jg$ has an assignment on
objects $(a, u) \mapsto (ha, ku)$ and an assignment on morphisms
$\langle 1_{a}, v \rangle \mapsto \langle 1_{ha}, kv \rangle$.
As we have 
$h\varphi(a, u) = \psi(ha, ku)$ by the definition of a morphism
of delta lenses, it follows that 
$hp(a, u) = \cod(h\varphi(a, u)) = \cod(\psi(ha, ku)) = q(ha, ku)$. 
A similar argument on morphisms of $Ef$ establishes that 
$q \circ J\langle h, k\rangle = h \circ p$ and thus the 
equation \eqref{equation:T-algebra-morphism} for a 
morphism of \JR-algebras holds. 
\end{proof}

\begin{proof}
We now construct a functor $\JRAlg(T, \nu) \rightarrow \Lens$ 
and show that it is inverse to $\Lens \rightarrow \JRAlg(T, \nu)$. 

Given a \JR-algebra determined by the pair of functors 
$f \colon A \rightarrow B$ and $p \colon Jf \rightarrow A$, 
we define a delta lens $(f, \varphi) \colon A \rightarrow B$ 
whose lifting operation $\varphi$ is given below, 
where $p(a, 1_{fa}) = a$ by \eqref{equation:JR-algebra}. 
\begin{equation}
\label{equation:algebra-to-lens}
	\begin{tikzcd}
    a
    \arrow[r, equal]
    &
    a
    \\[-2ex]
    fa
    \arrow[d, "1_{fa}"']
    \arrow[r, equal]
    &
    fa
    \arrow[d, "u"]
    \\
    fa
    \arrow[r, "u"]
    &
    b
    \end{tikzcd}
\qquad \longmapsto \qquad
	\begin{tikzcd}[column sep = huge]
	p(a, 1_{fa}) = a 
	\arrow[r, "{p\langle 1_{a},\, u\rangle}"]
	&
	p(a, u)
	\end{tikzcd}
\end{equation}
By \eqref{equation:JR-algebra} on morphisms, 
it follows that axiom \ref{(L2)} for a delta lens holds. 
By the left diagram of~\eqref{equation:semi-monad-algebra},
it is also immediate that axiom \ref{(L1)} holds. 
For axiom \ref{(L3)} to hold, we need to show that
\[
	p\langle 1_{a}, v \circ u \rangle 
	= p\langle 1_{a}, v \rangle \circ p \langle 1_{a}, u \rangle
	= p\langle 1_{p(a, u)}, v \rangle \circ p\langle 1_{a}, u \rangle.
\]
This amounts to proving that the morphism 
$p\langle 1_{a}, v \rangle \colon p(a, u) \rightarrow p(a, v \circ u)$
is equal to the morphism 
$p\langle 1_{p(a, u)}, v \rangle \colon p(p(a, u), 1_{b}) 
\rightarrow p(p(a, u), v)$, 
which follows directly from the right diagram in 
\eqref{equation:semi-monad-algebra}. 

Given a morphism of \JR-algebras 
$\langle h, k \rangle \colon (f, p) \rightarrow (g, q)$, 
we have $h p\langle 1_{a}, u \rangle = q\langle 1_{ha}, ku \rangle$ 
from~\eqref{equation:T-algebra-morphism}.
Therefore there is a well-defined 
morphism $\langle h, k \rangle$ between the corresponding delta lenses.

To show that the functors $\Lens \rightarrow \JRAlg(T, \nu)$ and 
$\JRAlg(T, \nu) \rightarrow \Lens$ are inverse, it is enough to 
show it holds on the objects as the morphisms consist of the same data.

First consider a delta lens $(f, \varphi)$ and define a functor 
$p \colon Jf \rightarrow B$ as in \eqref{equation:lens-to-algebra}. 
Applying this functor at a morphism 
$\langle 1_{a}, u \rangle \colon (a, 1_{fa}) \rightarrow (a, u)$ in
$Jf$, we obtain $\varphi(p(a, 1_{fa}), u) = \varphi(a, u)$ by 
\ref{(L2)} as desired.
Now consider a \JR-algebra $(f, p)$ and define a lifting operation 
$\varphi$ for a delta lens as in \eqref{equation:algebra-to-lens}. 
Defining a functor $\hat{p} \colon Jf \rightarrow A$ from this 
delta as in \eqref{equation:lens-to-algebra} and applying it to a 
morphism 
$\langle 1_{a}, v \rangle \colon (a, u_{1}) \rightarrow (a, u_{2})$ 
we find that $\hat{p}\langle 1_{a}, v \rangle = 
p\langle 1_{p(a,u)}, v \rangle = p\langle 1_{a}, v \rangle$ by the 
right diagram in \eqref{equation:semi-monad-algebra} as desired. 
This completes the proof. 
\end{proof}

\end{document}